\documentclass[preprint,12pt,authoryear,leqno]{elsarticle}
\usepackage{amsmath,amssymb,latexsym,amsthm,enumerate,booktabs}
\usepackage{mathtools}   
\usepackage[pagebackref]{hyperref}
\hypersetup{citecolor=red, linkcolor=blue, colorlinks=true}
\usepackage{todonotes,color,caption}

\usepackage{amsrefs}
\usepackage[right]{lineno}
\usepackage{nicefrac}

\journal{Finite Fields and their Applications}

\newtheorem{theorem}{Theorem}[section]
\newtheorem{lemma}[theorem]{Lemma}

\theoremstyle{definition}

\newtheorem{example}[theorem]{Example}
\newtheorem{problem}[theorem]{Problem}

\newtheorem{remark}[theorem]{Remark}

\renewcommand{\ge}{\geqslant}
\renewcommand{\le}{\leqslant}

\newcommand{\ABinom}[2]{\genfrac{\langle}{\rangle}{0pt}{}{#1}{#2}_{\kern -2pt q}}
\newcommand{\Binom}[2]{{\genfrac{[}{]}{0pt}{}{#1}{#2}}}
\newcommand{\binomq}[2]{{\binom{#1}{#2}}_{\kern-3pt q}}
\newcommand{\binomqq}[2]{{\binom{#1}{#2}}_{\kern-3pt -q}}

\newcommand{\CO}{\mathrm{CO}}
\newcommand{\GL}{{\mathrm{GL}}}
\newcommand{\GO}{{\mathrm{GO}}}

\newcommand{\Sp}{{\mathrm{Sp}}}

\newcommand{\Z}{{\mathbb Z}}

\newcommand{\cP}{{\mathcal{P}}}
\newcommand{\cN}{{\mathcal{N}}}
\newcommand{\cR}{\mathcal{R}}

\newcommand{\cT}{{\mathcal T}}
\newcommand{\cU}{{\mathcal{U}}}
\newcommand{\cW}{{\mathcal{W}}}

\newcommand{\eps}{\varepsilon}
\newcommand{\F}{{\mathbb F}}

\tikzset{every picture/.style={line width=0.11mm}}  
\newcommand{\oPerpSymbol}{\begin{tikzpicture}[scale=0.134]  
  \draw (0,-0.5)--(0,1); \draw (-0.866,-0.5)--(0.866,-0.5);
  \draw (0,0) circle [radius=1];\end{tikzpicture}}
\newcommand{\oPerp}{\mathbin{\raisebox{-1pt}{\oPerpSymbol}}}

\begin{document}
\begin{frontmatter}
\title{The probability of spanning a classical space by two  non-degenerate subspaces of complementary dimensions}
\author[add1]{S.P.~Glasby\texorpdfstring{\corref{cor}}{}}
\ead{Stephen.Glasby@uwa.edu.au}
\author[add1,add2]{Alice C.~Niemeyer}
\ead{Alice.Niemeyer@mathb.rwth-aachen.de}
\author[add1]{Cheryl E.~Praeger}
\ead{Cheryl.Praeger@uwa.edu.au}
\affiliation[add1]{organization=
  {Centre for the Mathematics of Symmetry and Computation, University of Western Australia  },
  addressline={35 Stirling Highway}, 
            city={Perth},
            postcode={6009}, 
            country={Australia.}
}
\affiliation[add2]{organization=
  {Chair for Algebra and Representation Theory,
    RWTH Aachen University},
  addressline={Pontdriesch 10-16},
            city={Aachen},
            postcode={52062}, 
            country={Germany.}
}
\cortext[cor]{Corresponding author}

\begin{abstract}
  Let $n,n'$ be positive integers and let $V$ be an
  $(n+n')$-dimensional vector space over a finite field $\F$ equipped
  with a non-degenerate alternating, hermitian or quadratic form.  We
  estimate the proportion of pairs $(U, U')$, where $U$ is a
  non-degenerate $n$-subspace and $U'$ is a non-degenerate $n'$-subspace
  of $V$, such that $U+ U'=V$
  (usually such spaces $U$ and $U'$ are not perpendicular). The
  proportion is shown to be at least $1-c/|\F|$ for some constant
  $c< 2$ in the symplectic or unitary cases, and $c<3$ in the
  orthogonal case.
\end{abstract}

\begin{keyword}
non-degenerate subspace, complement, finite classical group
\end{keyword}

\end{frontmatter}

\section{Introduction}\label{sec:intro}

Let $n, n'$ be positive integers and let $V$ be an
$(n+n')$-dimensional vector space over a finite field $\F$. The
proportion of pairs of subspaces $U, U'$ of dimensions $n$ and $n'$,
respectively, that intersect trivially, and hence satisfy $U+U'=V$, is
close to 1 if $|\F|$ is large and in all cases is greater than
$1-3/(2|\F|)$, see Lemma~\ref{L:VS}.  Our aim is to estimate this
proportion in the case where $V$ is a symplectic, hermitian or
orthogonal space and both $U$ and $U'$ are non-degenerate. Such an
analogue of Lemma~\ref{L:VS} seems not to have been considered before.

Configurations involving complementary subspaces in vector spaces
arise in a number of different ways; for example, they have been
studied graph theoretically \cites{Cla1992, JJ2016, Songpon},
geometrically \cite{IM1995}, and algorithmically \cite{PSY2015}.  For
different applications the types of subspaces considered are
restricted.  In representation theory they are usually submodules as
in \cite{Yar2013}.  Often the finite vector space $V$ is endowed with
a \textit{classical form}, that is, a non-degenerate sesqui-linear or
quadratic form, and in finite geometry the subspaces are usually
totally isotropic while in computational group theory they are
frequently non-degenerate.

The problem that we address in this paper arises from algorithmic
considerations connected with computations in finite classical
groups~\cite{random,stingray}. In order to show that two isometries,
each leaving invariant a non-degenerate
proper subspace, generate a classical group with high probability (see
for example \cites{DLLO13, PSY2015}), a fundamental problem arises.
Show that, with high probability, for a vector space $V$
endowed with a classical form, two random non-degenerate subspaces
whose dimensions sum to $\dim(V)$, are complements of each other, that
is, they intersect trivially and span $V$.  Our aim is to solve this
problem, and in fact to prove that the probability approaches~$1$ as
$|\F|$ approaches $\infty$. We emphasize that, for algorithmic applications,
an explicit lower bound for this probability, rather than an
asymptotic expression, is required.

\begin{theorem}\label{T:main1}
  Let $n, n'$ be positive integers, let $V=\F^{n+n'}$ be equipped with a
  non-degenerate alternating, hermitian or quadratic linear form, and
  let $c$ be a constant, with the type of form, $|\F|$ and $c$ as in
  Table~$\ref{tab1}$.  Then the proportion of pairs $(U, U')$ of
  non-degenerate subspaces of dimensions $n, n'$ respectively (of given
  type as in Table~$\ref{tab1}$ in the orthogonal case) that span
  $V$ is positive and at least $1-c/|\F|$.
\end{theorem}

Instead of proving a lower bound for  the pairs of subspaces we seek,
our strategy is  to determine an upper bound for the pairs of subspaces that intersect non-trivially.
In fact, many authors have considered problems concerning families of
subspaces which pairwise intersect non-trivially, for example,
generalisations of the Erd\H{o}s-Ko-Rado theorem  for sets, or the
Hilton-Milnor Theorem which gives upper bounds for the sizes of such families
of subspaces (see \cites{HMtheorem,GLW,Hsieh}  or related references).
Unfortunately we cannot exploit these results to prove our theorem: usually
our families of subspaces do not have the pairwise intersecting
property, and our families are orbits under a finite classical group. Moreover,
it is essential that the subspace dimensions sum to $\dim(V).$
A result \cite{DeBoeck}*{Lemma~4}
most closely related in spirit to ours determines the proportion 
of pairs of subspaces $(U,U')$ of a hermitian space $V$,  such that $U$ is
totally isotropic  of dimension $\dim(V)/2$, $U'$ is non-degenerate
of given dimension at most $\dim(V)/2$, and $U\cap U' = \{0\}.$
Other authors \cite{GW} have considered orbits of subspaces of finite
classical groups but study different aspects. In our proof we exploit the 
symmetry of classical geometry and employ the combinatorial technique
of ``double counting'', see Section~\ref{S:double}.

\begin{center}
\begin{table}[!ht]
  \caption{The constant $c$ in Theorem~\ref{T:main1}.}
  \centering
\begin{tabular}{lllll}  \toprule
  Case &Form & $|\F|$ & $c$ & Conditions \\
  \midrule
  unitary &hermitian & $q^2$ & $\frac{9}{5}$ & $q\ge2$, $n,n'\ge1$, $(n,n', q)\ne(1,1,2)$, \\[1.5mm]
  symplectic &alternating & $q$ &$\frac{5}{3}$ & $q\ge2$, and $n,n'\ge2$ are even,\\[1.5mm]
  orthogonal & {quadratic} & $q$ &$\frac{43}{16}$  
    & $q\ge3$, $n,n'\ge1$, $(n,n', q)\ne(1,1,3)$,\\
  &&&& if one of $n,n'$ is odd, then $q$ is odd.\\
  \bottomrule
\end{tabular} \label{tab1}
\end{table}
\end{center}

In the orthogonal case we require that $U$ and $U'$ lie in specified orbits
under the isometry group. Notation for the ``type'' of an orbit is given in
Remark~\ref{R:type}. Our methods are not strong enough to prove
Theorem~\ref{T:main1} in the orthogonal case when
$|\F|=2$ with a sufficiently small value of the constant $c$.
However, even though it is not covered by Theorem~\ref{T:main1}, 
extensive computer experimentation suggests that
the result holds in this case.

\begin{problem}
Show that Theorem~$\ref{T:main1}$ holds in the orthogonal case
with $|\F|=2$ (for some constant $c<2$).
\end{problem}

Theorem~\ref{T:main1} is a direct consequence of Theorems~\ref{T:U}, \ref{T:Sp}, and \ref{T:O}.

\section{Notation and strategy for proving Theorem~\texorpdfstring{\ref{T:main1}}{}} \label{S:strategy}\label{S:proof1}

In this section we establish the notation and hypotheses
used throughout this paper.  First, we prove the analogue of
Theorem~\ref{T:main1} mentioned in Section~\ref{sec:intro}.

\begin{lemma}\label{L:VS}
Let $q$ be a power of a prime and  let $n,  n'$ be  positive integers.
Let   $V=(\F_q)^{n+n'}$ be a vector space and let $\binom{V}{n}$ denote the
set of $n$-subspaces of $V$. Then the proportion $\rho$
of pairs $(U,U')$ in
$\binom{V}{n}\times\binom{V}{n'}$ with $U\cap U'=0$ satisfies
$\rho=\prod_{i=1}^{n'} \frac{1-q^{-i}}{1-q^{-n-i}}>1-3/(2q)$.
\end{lemma}

\begin{proof}
  As $\GL(V)$ is transitive on $\binom{V}{n}$, we fix some
  $U\in\binom{V}{n}$, and then $\rho$ is the proportion of
  $U'\in\binom{V}{n'}$ satisfying $U\cap U'=0$. There are $q^{nn'}$ complements
  $U'\in\binom{V}{n'}$ to~$U$ and
  \[
    \left|\binom{V}{n'}\right|
    =\frac{|\GL_{n+n'}(q)|}{|\GL_n(q)||\GL_{n'}(q)|q^{nn'}}
    =q^{nn'}\prod_{i=1}^{n'}\frac{1-q^{-n-i}}{1-q^{-i}}.
  \]
  Hence
  $\rho=q^{nn'}/\left|\binom{V}{n'}\right|
  =\prod_{i=1}^{n'}\frac{1-q^{-i}}{1-q^{-n-i}}$.
  Therefore  $\rho\ge \prod_{i=1}^{n'} (1-q^{-i})>\prod_{i=1}^\infty (1-q^{-i})$, and so
  $\rho>1-q^{-1}-q^{-2}\ge1-3/(2q)$ by~\cite{NP}*{Lemma~3.5}.
\end{proof}

Let $n,  n'$ be  positive integers and 
let $q$ be a power of a prime $p$. 
Let $V$  be an $(n+n')$-dimensional classical space over a field $\mathbb{F}$,
that is, a  non-degenerate symplectic or orthogonal
space over $\mathbb{F}=\F_q$, or a non-degenerate hermitian space over 
$\mathbb{F}=\F_{q^2}$. 

The isometry type of a subspace of $V$ depends only on its dimension in the 
symplectic and unitary cases, but is more complicated to
describe in the orthogonal case, see Remark~\ref{R:type}.
In the symplectic and unitary cases, let $\cU=\Binom{V}{n}$
(resp. $\cU'=\Binom{V}{n'}$) denote the set
of non-degenerate $n$-subspaces (resp. $n'$-subspaces) of $V$.
In the case of an orthogonal space of type $\eps$, let
$\cU=\Binom{V}{n}^\eps_\sigma$ and
$\cU'=\Binom{V}{n'}^\eps_{\sigma'}$ 
denote the set of non-degenerate $n$- and $n'$-subspaces of~$V$ with
subspace type $\sigma$  and $\sigma'$, 
 respectively, as in Remark~\ref{R:type}.

We will show that the 
proportion of pairs
$(U,U')\in \cU\times \cU'$ satisfying $U\cap U'=0$ is bounded away from zero, and is close to $1$ if $q$ is large.
The isometry group
$G(V)$ leaves invariant the set of  the pairs $(U,U')\in \cU \times \cU'$ satisfying $U\cap U'=0$,
and by Witt's Theorem, $G(V)$ is transitive on~$\cU$. We henceforth
fix a subspace $U\in \cU$ and note that
\[
\frac{|\{(U,U')\in \cU
  \times  \cU'\colon U\cap U'=0\}|}{|\cU\times \cU'|}=
  \frac{|\{U'\in \cU'\colon U\cap U'=0\}|}{|\cU'|}.
\]
Rather than showing that the above proportion is bounded away from zero, 
we find an upper bound, less than~$1$,  for the quantity
\begin{equation}\label{E:rho}
  \varphi\vcentcolon=\frac{|\{U'\in \cU'\colon U\cap U'\ne0\}|}{|\cU'|}.
\end{equation}
We use the following strategy to bound $\varphi$. If
$U\cap U'$ is not trivial, it must contain a $1$-subspace. Thus
\[
  \{U'\in \cU'\colon U\cap U'\ne0\}=\bigcup_{W}\;\{U'\in \cU'\colon W\le U\cap U'\},
\]
where $W$ ranges over the $1$-subspaces of $U$.
A $1$-subspace $W$ either lies in the set $\cN_1(U)$ of
non-degenerate $1$-subspaces of $U$, or the
set $\cP_1(U)$ of totally isotropic/singular $1$-subspaces of~$U$.
(We say that $W$ is \emph{totally isotropic} in the symplectic and unitary cases if $W\le W^\perp$ and \emph{totally singular} in the orthogonal case if $Q(W)=0$, which implies $W\le W^\perp$. Since $\dim(W)=1$ we sometimes omit the adjective ``totally''.)
For $W\le U$, let 
\[
  \cU'(W)=\{U'\in \cU'\colon W\le U'\}.
\]
Thus we have 
\[
  \{U'\in \cU'\colon U\cap U'\ne0\} = \left(\bigcup_{W\in\cP_1(U)} \cU'(W)\right)
  \cup \left(\bigcup_{W\in\cN_1(U)} \cU'(W)\right).
\]
Hence we have the upper bound
\begin{equation}\label{E:rho2}
  \varphi= \frac{|\{U'\in \cU'\colon U\cap U'\ne0\}|}{|\cU'|}
  {\le}\, \frac{|\bigcup_{W\in\cP_1(U)} \cU'(W)|}{|\cU'|} {+}
   \frac{|\bigcup_{W\in\cN_1(U)} \cU'(W)|}{|\cU'|}.
\end{equation}
  
In the symplectic case the set $\cN_1(U)$ is empty.
Given an arbitrary element $U'$ of $\cU$, we prove in
Lemma~\ref{T:IE} that
\begin{equation}\label{E:rhoSp}
  \varphi =\frac{\left|\bigcup_{W\in\cP_1(U)} \cU'(W)\right|}{|\cU'|} \le
  \frac{|\cP_1(U)||\cP_1(U')|}{|\cP_1(V)|}
\end{equation}
and moreover that $\varphi < 1.5/q$ if $q \ge 3.$ The exceptional case when $q=2$ is handled in the remainder of Section~\ref{S:Sp}, completing the proof
of Theorem~\ref{T:Sp} that $\varphi<5/(3q)$ for $n,n',q\ge2$.

We deal with the unitary and orthogonal cases in Sections~\ref{S:U} and~\ref{S:O} respectively. In both cases the isometry group $G(V)$ is transitive on the isotropic/singular $1$-subspaces of $V$ and we show, in the proof of Theorem~\ref{T:U} (see Equation~\eqref{E:Toad}) and Lemma~\ref{L:cP1} respectively, that 
\begin{equation}\label{E:rhoUO-P}
\frac{\left|\bigcup_{W\in\cP_1(U)} \cU'(W)\right|}{|\cU'|}
\le \frac{|\cP_1(U)||\cP_1(U')|}{|\cP_1(V)|}=\vcentcolon  \frac{c_2}{q|\mathbb{F}|}
\end{equation}
obtaining an explicit estimate for $c_2$.
In the unitary case $G(V)$ is also transitive on the non-degenerate
$1$-subspaces of~$V$ and we show in the proof of Theorem~\ref{T:U} (see Equation~\eqref{E:Frog}), that 
\begin{equation}\label{E:rhoU-N}
  \frac{ \left|\bigcup_{W\in\cN_1(U)} \cU'(W)\right|}{|\cU'|}
  \le  \frac{|\cN_1(U)||\cN_1(U')|}{|\cN_1(V)|}
  =\vcentcolon\frac{c_{1,\bf U}}{|\mathbb{F}|}
\end{equation}
and find an explicit estimate for $c_{1,\bf U}$.

The orthogonal case requires a more delicate analysis because the
isometry group $G(V)$ has at most two orbits on non-degenerate $1$-subspaces
of~$V$. For this reason, the definition~\eqref{E:rhoU-N} of $c_{1,\bf U}$
is adapted in the orthogonal case:
\begin{equation}\label{E:rhoO-N}
  \frac{ \left|\bigcup_{W\in\cN_1(U)} \cU'(W)\right|}{|\cU'|}
  \le   \frac{ \sum_{W\in\cN_1(U)} |\cU'(W)|}{|\cU'|}
  =\vcentcolon  \frac{c_{1,\bf O}}{|\mathbb{F}|}. 
\end{equation}
An intricate proof in  Lemma~\ref{L:cN1} shows that $c_{1,\bf O}$ is bounded
by explicit functions in $n, n', q$. Thus in both the unitary
and orthogonal cases we obtain  
\begin{equation}\label{E:rhoUO}
  \varphi \le \frac{c_{1,\bf U}}{|\F|} + \frac{c_2}{q|\F|}\qquad\textup{and}\qquad
  \varphi \le \frac{c_{1,\bf O}}{|\F|} + \frac{c_2}{q|\F|}. 
\end{equation}
This allows us to prove that $\varphi < c/|\mathbb{F}|$ for a
small constant $c$, see Theorems~\ref{T:U} and~\ref{T:O}.
In the orthogonal case, our proof needs $q>2$
in order to obtain $c<|\F|$ and hence $\varphi<1$.
(In the unitary and orthogonal sections, namely Sections~\ref{S:U} and~\ref{S:O},
we write $c_1$ rather than $c_{1,\bf U}$ and $c_{1,\bf O}$ as the meaning
is clear.)

\section{Double counting lemma}\label{S:double}

The following simple result is surprisingly effective in our numerous
different contexts when solving estimation problems. 
Let $\cW$ and $\cU'$ be sets of subspaces of~$V$. Define

\begin{align*}
    \cU'(W)&=\{U'\in \cU'\mid W\le U'\}\quad\,\textup{for fixed $W\in \cW$, and}\\
    \cW(U')&=\{W\in \cW\mid W\le U'\}\quad\textup{for fixed $U'\in \cU'$}.
\end{align*}

\begin{lemma}[Double counting lemma]\label{L:DC}
  Let $G\le\GL(V)$ and let $\cW$ and $\cU'$ be $G$-orbits on the set of all
  subspaces of~$V$.
  For fixed $W\in \cW$, $U'\in \cU'$ and $\cU'(W), \cW(U')$ as above, we have
  \[
  |\cW|\cdot|\cU'(W)|=|\cW(U')|\cdot|\cU'|.
  \]
\end{lemma}

\begin{proof}
  Let $\Omega=\{(W,U')\in \cW\times \cU'\mid W\le U'\}$. Observe that
  \begin{equation*}
    \{(W,U')\mid W\in \cW, U'\in \cU'(W)\}=
    \Omega=\{(W,U')\mid U'\in \cU', W\in \cW(U')\}.
  \end{equation*}
  Since $\cW$ and $\cU'$ are $G$-invariant, it is easy to show 
  for all $g\in G$ and fixed $(W,U')\in\cW\times\cU'$ that $\cU'(W)^g=\cU'(W^g)$
  and $\cW(U')^g=\cW((U')^g)$. Thus we have
  $|\cU'(W^g)|=|\cU'(W)|$ and $|\cW((U')^g)|=|\cW(U')|$ for all $g\in G$.
  It follows from  the above display  that
  $|\cW|\cdot|\cU'(W)|=|\Omega|=|\cW(U')|\cdot|\cU'|$
  for any chosen $(W,U')\in \cW\times \cU'$, as claimed.
\end{proof}

\section{The unitary case}\label{S:U}

Let $V=(\F_{q^2})^{n+n'}$ be a non-degenerate hermitian space where $n,n'\ge1$, and let
$\cU=\Binom{V}{n}$ and $\cU'=\Binom{V}{n'}$, as in Section~\ref{S:proof1}.
In this section we prove Theorem~\ref{T:main1} in the unitary case by establishing the
following result which determines an upper bound for the proportion $\varphi$ in~\eqref{E:rho}.

\begin{theorem}\label{T:U}
  Suppose that $n,n'\ge1$ and $V=(\F_{q^2})^{n+n'}$ is a non-degenerate hermitian
  space. Let $\cU=\Binom{V}{n}$ and $\cU'=\Binom{V}{n'}$,
  and fix $U\in \cU$. Then
  \[
  \varphi=\frac{|\{U'\in \cU'\mid U\cap U'\ne0\}|}{|\cU'|}
  \le\frac{|\cN_1(U)||\cN_1(U')|}{|\cN_1(V)|}
  +\frac{|\cP_1(U)||\cP_1(U')|}{|\cP_1(V)|}
  \le\frac{9}{5q^2}
  \]
  holds if $(n,n',q)\ne(1,1,2)$, and $\varphi\le\frac{2}{q^2}$ always holds.
\end{theorem}

We need the following technical lemma which estimates
certain polynomials in $-q^{-1}$.

\begin{lemma}\label{L:ineq}
    Suppose $\zeta_n\vcentcolon=\zeta_n(q)=(1-(-q)^{-n+1})(1-(-q)^{-n})$ and $q\ge2$. Then
  \begin{enumerate}[{\rm (a)}]
  \item 
    $0=\zeta_1<\zeta_3<\zeta_5<\cdots<1<\cdots<\zeta_6<\zeta_4<\zeta_2=(1+q^{-1})(1-q^{-2})$,
  \item $\zeta_{2i+2}\zeta_{2j}<\zeta_{2i}\zeta_{2j+2}$ for $1\le i< j$,
  \item $\zeta_{n}\zeta_{n'}\zeta_{4}\le\zeta_{2}^2\zeta_{n+n'}$ for $n,n'\ge1$.
  \end{enumerate}
\end{lemma}

\begin{proof}
  (a)~Suppose that $i\ge1$. By definition, the inequality
  $\zeta_{2i-1}<\zeta_{2i+1}$ is
  \[
    (1-q^{-2i+2})(1+q^{-2i+1})<(1-q^{-2i})(1+q^{-2i-1}).
  \]
  After subtracting $1$  from both sides, multiplying by $q^{4i+1}$,
  and rearranging,  this is equivalent to
  \[
    q^{2i+2}+q^{2i+1}+1<q^{2i+3}+q^{2i}+q^4.
  \]
  The stronger inequality $q^{2i+2}+q^{2i+1}+1<q^{2i+3}$ is true for
  $q\ge2$. Therefore $\zeta_{2i-1}<\zeta_{2i+1}$ holds, and
  since $\lim_{i\to\infty}\zeta_{2i+1}=1$ we see that
  $0=\zeta_1<\zeta_3<\zeta_5<\cdots<1$ holds.

  Similarly, by definition, the inequality $\zeta_{2i+2}<\zeta_{2i}$ is
  \[
    (1+q^{-2i-1})(1-q^{-2i-2})<(1+q^{-2i+1})(1-q^{-2i}),
  \]
  and after subtracting $1$ from both sides, multiplying by $q^{4i+3}$,
  and rearranging, this is equivalent to
  \[
    q^{2i+3}+q^{2i+2}+q^4<q^{2i+4}+q^{2i+1}+1.
  \]
  Since, $i\ge1$ and $q\ge2$, we have
  \[
    q^{2i+3}+q^{2i+2}+q^4\le q^{2i}(q^3+q^2+q^2)\le q^{2i+4}
    <q^{2i+4}+q^{2i+1}+1.
  \]
  Hence $\zeta_{2i+2}<\zeta_{2i}$ holds, and since
	$\lim_{i\to\infty}\zeta_{2i}=1$ we see that
  $1<\cdots<\zeta_6<\zeta_4<\zeta_2=(1+q^{-1})(1-q^{-2})$ holds, proving part (a).

  (b)~Observe first that
  $\zeta_{2i}=q^{-4i+1}(q^{2i-1}+1)(q^{2i}-1)$. Replacing $i$ with $i+1$ gives
  $\zeta_{2i+2}=q^{-4i-3}(q^{2i+1}+1)(q^{2i+2}-1)$. 
  Hence $\zeta_{2i}\zeta_{2j+2}-\zeta_{2i+2}\zeta_{2j}$ equals
  $q^{-4i-4j-2}x$ where $x$ is the following difference
  \[
  (q^{2i-1}{+}1)(q^{2i}{-}1)(q^{2j+1}{+}1)(q^{2j+2}{-}1)
                   -(q^{2i+1}{+}1)(q^{2i+2}{-}1)(q^{2j-1}{+}1)(q^{2j}{-}1).
  \]
  {\sc Mathematica}~\cite{Wolfram} shows that
  $x=q^{-1}(q^2-1)(q^{2j}-q^{2i})y$ where
  \[
    y = q^{2i+2j+2}-q^{2i+2j+1}-q^{2j+2}-q^{2j}-q^{2i+2}-q^{2i}-q+1.
  \]
  Suppose that $1\le i<j$. Proving
  $\zeta_{2i}\zeta_{2j+2}-\zeta_{2i+2}\zeta_{2j}>0$ is equivalent to proving $x>0$,
  which is equivalent to proving $y>0$. If the exponents of $q$ are
  distinct, that is if $i+1<j$, then this is clear. If $i+1=j$, then
  $y = q^{2i+2j+2}-q^{2i+2j+1}-q^{2j+2}-2q^{2j}-q^{2i}-q+1$ and $2q^{2j}\le q^{2j+1}$
  and it is also clear that $y>0$.
  Consequently, $\zeta_{2i+2}\zeta_{2j} < \zeta_{2i}\zeta_{2j+2}$ as claimed.

  (c)~If $n$ is even and $n'$ is odd, then part~(a) shows that
  $\zeta_{n}\le\zeta_{2}$, $\zeta_{n'}<\zeta_{n+n'}$, $\zeta_{4}<\zeta_{2}$ and
  hence $\zeta_{n}\zeta_{n'}\zeta_{4}\le\zeta_{2}^2\zeta_{n+n'}$.
  The case when $n$ is odd and $n'$ is even is proved similarly.
  If both $n$ and $n'$ are odd, then
  $\zeta_{n'}<\zeta_{2}$, $\zeta_{n}<\zeta_{n+n'}$ and $\zeta_{4}<\zeta_{2}$
  by part~(a),
  and hence $\zeta_{n}\zeta_{n'}\zeta_{4}\le\zeta_{2}^2\zeta_{n+n'}$.
  Finally, suppose that both $n$ and $n'$ are even.
  By symmetry,  we may assume that $n\le n'$. The inequality holds trivially 
  if $n=n'=2$, so we may assume in addition that $n'\ge 4$. 
  Then $\zeta_{4}\zeta_{n'}\le\zeta_{2}\zeta_{n'+2}$ by part~(b)  
  (with $i=1$ and $j=n'/2$).
  Hence $\zeta_{n}\zeta_{n'}\zeta_{4}\le\zeta_{2}^2\zeta_{n+n'}$ holds when $n=2$.
  Henceforth assume that $4\le n\le n'$.
  Then  $\zeta_{n}\zeta_{n'+2}<\zeta_{n-2}\zeta_{n'+4}<\cdots<\zeta_{2}\zeta_{n+n'}$
  by repeated application of part~(b). Multiplying
  $\zeta_{n}\zeta_{n'+2}\le\zeta_{2}\zeta_{n+n'}$ and
  $\zeta_{4}\zeta_{n'}\le\zeta_{2}\zeta_{n'+2}$ and cancelling $\zeta_{n'+2}$ gives
  $\zeta_{n}\zeta_{4}\zeta_{n'}  \le\zeta_{2}^2\zeta_{n+n'}$ as claimed.
 \end{proof}

\begin{proof}[Proof of Theorem~$\ref{T:U}$]
  Let $V=(\F_{q^2})^{n+n'}$, $\cU=\Binom{V}{n}$, $\cU'=\Binom{V}{n'}$ and
  fix $U\in\cU$.
  The isometry group $G=G(V)$ has two orbits on the 1-subspaces of $V$,
  namely the set of isotropic 1-subspaces $\cP_1(V)$ and the set of non-degenerate 1-subspaces
  $\cN_1(V)$. Similarly, $G_U$ has two orbits on the 1-subspaces of $U$,
  namely $\cP_1(U)$ and $\cN_1(U)$. Fix $W_\cN\in\cN_1(U)$ and
  $W_\cP\in\cP_1(U)$. Then
  \begin{align*}
    \left|\cup_{W\in\cN_1(U)\cup\cP_1(U)}\,\cU'(W)\right|
    &\le\left|\cup_{W\in\cN_1(U)} \cU'(W)\right|
       +\left|\cup_{W\in\cP_1(U)} \cU'(W)\right|\\
    &\le|\cN_1(U)|\cdot|\cU'(W_\cN)|+|\cP_1(U)|\cdot|\cU'(W_\cP)|.
  \end{align*}
  We apply the double counting lemma (Lemma~\ref{L:DC}) with $G=G(V)$
  and a fixed $U'\in\cU'=\Binom{V}{n'}$,
  first with  $\cW=\cN_1(V)$ to obtain
  $|\cN_1(V)|\cdot|\cU'(W_\cN)|=|\cN_1(U')|\cdot|\cU'|$, and then
  with $\cW=\cP_1(V)$ to obtain
  $|\cP_1(V)|\cdot|\cU'(W_\cP)|=|\cP_1(U')|\cdot|\cU'|$. Hence, substituting
  into the above display yields
  \[
    \varphi= \frac{\left|\bigcup_{W\in\cN_1(U)\cup\cP_1(U)}\cU'(W)\right|}{|\cU'|}
    \le\frac{|\cN_1(U)||\cN_1(U')|}{|\cN_1(V)|}
      +\frac{|\cP_1(U)||\cP_1(U')|}{|\cP_1(V)|}.
  \]
  A formula for $|\cP_1(U)|$ is given in~\cite{Taylor}*{10.4, p.\,117}
  and since $|\cN_1(U)|+|\cP_1(U)|$ equals the number $(q^{2n}-1)/(q^2-1)$
  of 1-subspaces of $V$, we have
  \begin{align*}
    |\cN_1(U)|&=\frac{q^{2n-2}(1-(-q)^{-n})}{1+q^{-1}} \quad\textup{and}\\
    |\cP_1(U)|&=\frac{q^{2n-3}(1-(-q)^{-n+1})(1-(-q)^{-n})}{1-q^{-2}}.
  \end{align*}
  Using the convenient notation $\vartheta_n=1-(-q)^{-n}$ and
  $\zeta_n=\vartheta_{n-1}\vartheta_n$, these expressions become 
  \begin{equation}\label{E:N1P1}
  |\cN_1(U)|=\frac{q^{2n-2}\vartheta_n}{1+q^{-1}}, \quad\textup{and}\quad
  |\cP_1(U)|=\frac{q^{2n-3}\zeta_n}{1-q^{-2}}.
  \end{equation}
  Thus
  \begin{align}\label{E:Frog}
    \frac{\left|\cup_{W\in\cN_1(U)} \cU'(W)\right|}{|\cU'|}
    &\le\frac{|\cN_1(U)|\cdot|\cN_1(U')|}{|\cN_1(V)|}=\vcentcolon\frac{c_1}{q^2}
    \quad\textup{where by~\eqref{E:N1P1},}\\
    c_1&=\frac{\vartheta_n\vartheta_{n'}}
        {(1+q^{-1})\vartheta_{n+n'}},\notag
  \end{align}
  and
  \begin{align}\label{E:Toad}
  \frac{\left|\cup_{W\in\cP_1(U)} \cU'(W)\right|}{|\cU'|}
    &\le\frac{|\cP_1(U)|\cdot|\cP_1(U')|}{|\cP_1(V)|}=\vcentcolon\frac{c_2}{q^3}
    \quad\textup{where by~\eqref{E:N1P1},}\\
    c_2&=\frac{\zeta_n\zeta_{n'}}{(1-q^{-2})\zeta_{n+n'}}.\notag
  \end{align}
  Thus $\varphi\le (c_1+c_2/q)q^{-2}$. We now bound $c_1+c_2/q$.
  
  First suppose that $n=1$. Then $\vartheta_1=1+q^{-1}$ and $\zeta_1=0$
  implies that $c_1=\vartheta_{n'}/\vartheta_{n'+1}$ by~\eqref{E:Frog} and $c_2=0$
  by~\eqref{E:Toad}. Therefore, when $n=1$, 
  \[
    c_1+\frac{c_2}{q}=c_1=\frac{\vartheta_{n'}}{\vartheta_{n'+1}}.
  \]
  If $n'$ is even then $\theta_{n'}<1<\theta_{n'+1}$ so $c_1<1<\frac{9}{5}$.
  On the other hand if $n'$ is odd, then
  $c_1=\frac{1+q^{-n'}}{1-q^{-n'-1}}\le\frac{1+q^{-1}}{1-q^{-2}}
  =\frac{1}{1-q^{-1}}<\frac{9}{5}$ if $q\ge3$.
  If $(n,n',q)=(1,1,2)$, we have $c_1+\frac{c_2}{q}=\frac{1+2^{-1}}{1-2^{-2}}=2$.
  Hence, for  $n=1$, $c_1+\frac{c_2}{q}\le \frac{9}{5}$ when
   $(n,n',q)\ne(1,1,2)$, and $c_1+\frac{c_2}{q}$ is always at most $2$, so
  Theorem~\ref{T:U} is proved in this case.
  As the expressions for $c_1$ and $c_2$ are symmetric in $n,n'$, 
  it follows that Theorem~\ref{T:U} is also proved when $n'=1$.

  Thus we may  (and will) assume henceforth that $n,n'\ge2$.
  We first prove that $c_1 <1$ using the expression in \eqref{E:Frog}.
  Since $n,n'\ge2$, we have
  \[
  \frac34\le 1-q^{-2}=\vartheta_2<\vartheta_4<\vartheta_6<\cdots<1<\cdots<\vartheta_5<\vartheta_3=1+q^{-3}\le \frac98,
  \]
  and hence $1-q^{-2}\le\vartheta_{n},\vartheta_{n'}\le 1+q^{-3}$.
  If $n,n'$ have different parities, then
  \[
  c_1<\frac{1\cdot\vartheta_3}{(1+q^{-1})\cdot 1}=\frac{1+q^{-3}}{1+q^{-1}} = 1-q^{-1} + q^{-2} < 1.
  \]
  If $n,n'$ are both even, then
  \[
  c_1<\frac{1^2}{(1+q^{-1})\cdot \vartheta_4}
    =\frac{1}{(1+q^{-1}) (1-q^{-4})} < 1.
  \]
  If $n,n'$ are both odd, then both $n,n'\ge3$, and so
  \[
  c_1\le \frac{\vartheta_3^2}{(1+q^{-1})\cdot \vartheta_6} 
  =\frac{(1+q^{-3})^2}{(1+q^{-1})\cdot(1-q^{-6})}
  =\frac{1-q^{-1}+q^{-2}}{1-q^{-3}}<1.
  \]
  In summary, for $n,n'\ge2$, we have
  $c_1<1$ as claimed.
  
  We now show that $c_2\le\frac{8}{5}$ when $n,n'\ge2$.
  By~\eqref{E:Toad} and Lemma~\ref{L:ineq}(c),
  \begin{align*}
  c_2 &= \frac{\zeta_n\zeta_{n'}}{(1-q^{-2})\zeta_{n+n'}}
  \le\frac{\zeta_2^2}{(1-q^{-2})\zeta_{4}}
  =\frac{(1+q^{-1})\zeta_2}{\zeta_4}\\
  &=\frac{(1+q^{-1})^2(1-q^{-2})}{(1+q^{-3})(1-q^{-4})}
  =\frac{1+q^{-1}}{(1-q^{-1}+q^{-2})(1+q^{-2})}.
  \end{align*}
  The last expression on the right above is a decreasing function of $q$. Setting $q=2$ shows
  $c_2\le \frac85$. Thus $c_1+c_2/q<1+4/5=9/5$ when $n,n'\ge2$.

  Thus we have proved, for all $n,n'\ge1$, $q\ge2$ and $(n,n',q)\ne(1,1,2)$, that
  \[
    \varphi \le\frac{|\cN_1(U)|\cdot|\cN_1(U')|}{|\cN_1(V)|}
      +\frac{|\cP_1(U)|\cdot|\cP_1(U')|}{|\cP_1(V)|}
      =\frac{c_1}{q^2}+\frac{c_2}{q^3}=\frac{c_1+c_2/q}{q^2}
      \le\frac{9}{5q^2},
  \]
  and when  $(n,n',q)=(1,1,2)$, then $\varphi \le  \frac{2}{q^2}$.
  This completes the proof.
\end{proof}

\section{The symplectic case}\label{S:Sp}

Let $V=(\F_q)^{n+n'}$ be a non-degenerate symplectic space where $n,n'$ are even, and
let $\cU=\Binom{V}{n}$ and $\cU'=\Binom{V}{n'}$, as in Section~\ref{S:proof1}.
In this section we prove Theorem~\ref{T:main1} in the symplectic case by establishing the following result.

\begin{theorem}\label{T:Sp}
 Suppose that $n,n'$ are even and $V=(\F_q)^{n+n'}$ is a non-degenerate symplectic
  space. Let $\cU=\Binom{V}{n}$ and $\cU'=\Binom{V}{n'}$, and fix $U\in \cU$. Then
  \[
  \varphi\vcentcolon= \frac{|\{U'\in \cU'\colon U\cap U'\ne0\}|}{|\cU'|}<\frac{5}{3q}
  \qquad\textup{for all $n,n',q$.}
  \]
 Furthermore if either $q\ge3$ or 
 $\min\{n,n'\}=2$, then $\varphi< \frac{3}{2q}$. 
\end{theorem}

As every 1-subspace of $U$ is totally isotropic, $\cN_1(U)$~is~empty.
Lemma~\ref{T:IE} below proves Equation~\eqref{E:rhoSp}, and so justifies
the last sentence of Theorem~\ref{T:Sp}. In particular it proves
Theorem~\ref{T:Sp} for $q\ge 3$.

\begin{lemma}\label{T:IE}
  Let $V, n, n', \cU, \cU'$ and $\varphi$ be as in Theorem~$\ref{T:Sp}$. Then
  \[
  \varphi\le\frac{|\cP_1(U)|\cdot|\cP_1(U')|}{|\cP_1(V)|}
  <\frac{1-q^{-\min\{n,n'\}}}{q-1},
  \]
  for all $U\in\cU$ and $U'\in\cU'$.
  Moreover, if $q\ge3$, or if $\min\{n,n'\}=2$, then $\varphi<\frac{3}{2q}< \frac{5}{3q}$.
\end{lemma}

\begin{proof}
  As all the 1-subspaces of $V$ are isotropic,
  $|\cP_1(V)|=(q^{n+n'}-1)/(q-1)$ and 
  $\varphi =  |\bigcup_{W'\in\cP_1(U)}\cU'(W')|/|\cU'|$.
  Since the restriction of $G_U$ to $U$ is $G(U)$, and $G(U)$ is
  transitive on $\cP_1(U)$, we have, for a chosen $W\in\cP_1(U)$,
  \[
  \varphi\cdot |\cU'| = \left|\bigcup_{W'\in\cP_1(U)}\cU'(W')\right|
  \le|\cP_1(U)|\cdot|\cU'(W)|.
  \] 
  Applying Lemma~\ref{L:DC} with $G$ the isometry group $\Sp(V)$ of $V$,
  $\cU'=\Binom{V}{n'}$, $\cW=\cP_1(V)$ and fixed $U'\in\cU'$, $W\in\cW$,
  we obtain   $|\cP_1(V)|\cdot |\cU'(W)|= |\cP_1(U')|\cdot |\cU'|$. Thus
  \begin{align*}
  \varphi&\le\frac{|\cP_1(U)|\cdot|\cU'(W)|}{|\cU'|}
      =\frac{|\cP_1(U)|\cdot|\cP_1(U')|}{|\cP_1(V)|}
      =\frac{(q^n-1)(q^{n'}-1)}{(q-1)(q^{n+n'}-1)}\\
    &=\frac{(1-q^{-n})(1-q^{-n'})}{(q-1)(1-q^{-n-n'})}
    <\frac{1-q^{-\min\{n,n'\}}}{q-1}.
  \end{align*}
  If $q\ge3$, then $\varphi<\frac{1}{q-1}\le\frac{3}{2q}<\frac{5}{3q}$.
  Finally, if $\min\{n,n'\}=2$, then
  \[
  \varphi<\frac{1-q^{-2}}{q-1}
  =\frac{1-q^{-2}}{q(1-q^{-1})}=\frac{1+q^{-1}}{q}\le\frac{3}{2q}<\frac{5}{3q}.\qedhere
  \]
\end{proof}

\subsection{The symplectic case  with \texorpdfstring{$q=2$}{}}\label{s:sp2}
In order  to prove Theorem~\ref{T:Sp} for $q=2$ we require a much more 
delicate argument. Henceforth assume that the hypotheses of
Theorem~\ref{T:Sp} hold, that $q=2$ and that $\min\{n,n'\}\ge4$.
For a subspace $W$ of $V$ recall that
$\cU'(W)=\{U'\in\cU'\mid W\le U'\}$ as in Lemma~\ref{L:DC},
and for a fixed $U\in\cU$ write
%
%
%
\begin{equation}\label{e:R}
\cT = \bigcup_{W\in\cP_1(U)}\cU'(W).
\end{equation}
We need to improve the
upper bound for $\varphi=|\cT|/|\cU'|$ given in Lemma~\ref{T:IE}.
For a positive integer $k$, let $\cP_k(U)$ denote the set of
totally isotropic $k$-subspaces of $U$. (Clearly
$|\cP_k(U)|=0$ for $k>\frac12\dim(U)$.) It follows
from~\cite{Taylor}*{Exercise 8.1(ii)} that
\begin{equation}\label{E:Pkq}
  |\cP_k(U)|=\prod_{i=0}^{k-1}\frac{q^{\dim(U)-2i}-1}{q^{k-i}-1}.
\end{equation}
To estimate $|\cT|$ we use the facts that the
symplectic group $G(U)$ on $U$ (which is induced by $G(V)_U$) has two
orbits on the set of $2$-subspaces of $U$, namely $\cP_2(U)$ and $\Binom{U}{2}$,
and that $G(U)$ is transitive on the set of  $3$-subspaces of $U$ which are not totally isotropic (the latter assertion is justified after \eqref{E:Part}).
 
  The goal of the next result (Lemma~\ref{L:Sp2}) is to establish a
  new upper bound
  $|\cT|\le T_1-T_2+T_3$, with the $T_k$ as in \eqref{E:Ti}, using
  the inclusion-exclusion
  principle. We compute $T_1$ directly, and compute $T_2$ and $T_3$ in terms
  of the quantities $S_2$ and $S_3$ defined in \eqref{E:Si}. Then we
  determine~$S_2$ and $S_3$ using the double counting Lemma~\ref{L:DC}.
  Before proceeding, we need some additional notation.

  Let $\binom{V}{k}$ denote the set of all $k$-subspaces of $V$. Consider
  the partition
  \begin{equation}\label{E:Part}
    \binom{V}{k}=\cP_k(V)\,\dot{\cup}\,\cR_k(V)
    \qquad\textup{where $\cR_k(V)=\binom{V}{k}\setminus\cP_k(V)$.}
  \end{equation}
  Note in particular that $\cR_2(V)=\Binom{V}{2}$. Each $W\in\cR_3(V)$
  is 3-dimensional and hence is  degenerate. Since $W\not\le W^\perp$
  and $W/(W\cap W^\perp)$ is non-degenerate, we have
  $W\cap W^\perp=\langle f'\rangle$ and $W$ is isometric
  to $\langle e,f,f'\rangle$ where
  $\langle e,f\rangle\in\Binom{V}{2}$.
  Hence $\cR_3(V)$ is an orbit under $G\vcentcolon=\Sp(V)$. Also each of
  $\cR_2(V)$, $\cP_2(V)$, and $\cP_3(V)$ is a $G$-orbit.
  Applying~\eqref{E:Part} to $\binom{U}{k}$ gives, for  $k\in\{2,3\}$,
  \begin{align}\label{E:Si}
  S_k &\vcentcolon=\sum_{W\in\binom{U}{k}}|\cU'(W)| =  S_{\cP,k} + S_{\cR,k} \quad \mbox{where}\\
   S_{\cP,k} &\vcentcolon=\sum_{W\in\cP_k(U)}|\cU'(W)|, \quad \mbox{and}\quad  
   S_{\cR,k} \vcentcolon= \sum_{W\in\cR_k(U)}|\cU'(W)|.\notag
  \end{align}

\begin{lemma}\label{L:Sp2}
   Assume that $\min\{n,n'\}\ge4$ and $q=2$. Given a positive integer $k$, let
  \begin{equation}\label{E:Ti}
    T_k=\sum_{ \{W_1,\dots,W_k\} } |\cU'(W_1+\cdots+W_k)|,
  \end{equation}
  where the sum is over all $k$-element subsets $\{W_1,\dots,W_k\}$
  of $\cP_1(U)$.
\begin{enumerate}[{\rm (a)}]
\item  Then $|\cT|\le T_1-T_2+T_3$ and, for $W\in\cP_1(U)$, 
  \[
    T_1=|\cP_1(U)|\cdot|\cU'(W)|\quad\textup{and satisfies}\quad
    \frac{T_1}{|\cU'|}=\frac{(1-2^{-n})(1-2^{-n'})}{1-2^{-n-n'}}<1.
  \]
\item Let $k\in\{2,3\}$. Then $\cR_k\ne\emptyset$ for all $n\ge4$,  and
  $\cP_k\ne\emptyset$ if and only if $(n,k)\ne(4,3)$. Further,
  $S_k= S_{\cP,k}+S_{\cR,k}$ as in~\eqref{E:Si} where $S_{\cP,k}=0$ if $(n,k)=(4,3)$,
  and otherwise
\[
  S_{\cP,k}= |\cP_k(U)|\cdot|\cU'(W_{\cP,k})|,\qquad
  S_{\cR,k}=\left|\cR_k(U)\right|\cdot |\cU'(W_{\cR,k})|,
\]
for chosen $W_{\cP,k}\in \cP_k(U)$ and $W_{\cR,k}\in \cR_k(U)$.
Moreover, 
\[
|\cT|\le  T_1-2S_2+28S_3.
\]
\end{enumerate}
\end{lemma} 
 
\begin{proof}
(a)  It is straightforward to see that, for $W_1,\dots, W_k\in\cP_1(U)$, 
  \begin{equation}\label{E:int}
    \cU'(W_1)\cap\cdots\cap\cU'(W_k)=\cU'(W)\quad\textup{where}\quad
    W=W_1+\cdots+W_k.  
  \end{equation}
  By the inclusion-exclusion principle, see for example~\cite{FS}*{III.7.4},
  we have
  \[
    |\cT|=\sum_{k\ge1}(-1)^{k-1}M_k\quad\textup{where}\quad
    M_k=\sum_{ \{W_1,\dots,W_k\}} |\cU'(W_1)\cap\cdots\cap\cU'(W_k)|
  \]
  and the sum for $M_k$ is over all $k$-element subsets of $\cP_1(U)$.
  It follows from \eqref{E:int} that $M_k =T_k$ as in \eqref{E:Ti}, 
  and that $T_k=0$ if
  $k>|\cP_1(U)|=(q^n-1)/(q-1)$.   Finally, we truncate the summation
  to obtain upper/lower bounds for $|\cT|$ such as
  $T_1-T_2\le T_1-T_2+T_3-T_4\le\cdots\le|\cT|\le\cdots\le T_1-T_2+T_3\le T_1$.
  In particular, $|\cT|\le T_1-T_2+T_3$.
   
  Since the restriction of $G(V)_U$ to $U$ is $G(U)$, and $G(U)$ is
  transitive on $\cP_1(U)$, it follows that $T_1=|\cP_1(U)|\cdot |\cU'(W)|$ for
  a fixed 1-subspace $W$ of $U$. It was
  shown in the proof of Lemma~\ref{T:IE} that
  \[
    \frac{|\cP_1(U)|\cdot|\cU'(W)|}{|\cU'|}
    = \frac{(1-q^{-n})(1-q^{-n'})}{(q-1)(1-q^{-n-n'})}.
  \] 
  Since we are assuming $q=2$, substituting $q=2$ gives the expression for $T_1/|\cU'|$ in part (a).

  (b)   Suppose that $k\in\{2,3\}$. Since $U$ is non-degenerate, it follows
  from~\eqref{E:Part} that   $\cR_k(U)\ne\emptyset$. 
  Observe that $\cP_k(U)=\emptyset$
  if and only if $k>n/2$. Hence, since $n\ge4$,  $\cP_k(U)=\emptyset$
  if and only if $(n,k)=(4,3)$, and so  $S_{\cP,k}=0$ if $(n,k)=(4,3)$. 
  For $(n,k)\ne(4,3)$, in the sums $S_k=S_{\cP,k}+S_{\cR,k}$, the contribution $|\cU'(W)|$ for a
  subspace $W$ of $U$ is the same for each $G_U$-image of $W$, and as we
  noted above there are at most two non-empty $G_U$-orbits on
  $k$-subspaces of $U$  for $k\in\{2,3\}$. Thus using the partition
  $\binom{U}{k}=\cP_k(U)\,\dot{\cup}\,\cR_k(U)$ in \eqref{E:Part}, 
  and choosing subspaces $W$ from each non-empty $G_U$-orbit we obtain
  the expressions for $S_{\cP,k}$ and $S_{\cR,k}$ given in part (b).
  
  Finally we prove the upper bound for $|\cT|$. 
  As $q=2$, each $2$-subspace $W$ of $U$ arises

  as $W=W_1+W_2$ for exactly three pairs $\{W_1,W_2\}$ from $\cP_1(U)$ and hence 
  $T_2= 3S_2$ with $S_2$ as in \eqref{E:Si}.
  Similarly the definition of $T_3$ involves a sum over unordered triples
  $\{W_1,W_2, W_3\}$ from $\cP_1(U)$. If $W=W_1+W_2+W_3$ is a $2$-subspace,
  then $W$ arises exactly once and the contribution to $T_3$ of such triples
  is precisely~$S_2$. On the other hand, if $W$ is a $3$-subspace, then $W$
  arises once for each of its $|\GL_3(2)|/3!=28$ spanning unordered triples
  of $1$-subspaces. Thus $T_3= S_2+28S_3$ with $S_2, S_3$ as in \eqref{E:Si}.
  Thus by part (a), $|\cT|\le T_1-T_2+T_3=  T_1-2S_2+28S_3$, as in part (b).
  %
\end{proof} 

\begin{lemma}\label{L:Sp4}
  Suppose that $q=2$, $\min\{n,n'\}\ge 4$ and $k\in\{2,3\}$. Then
  \begin{enumerate}[{\rm (a)}]
  \item
    $\displaystyle
    \frac{|S_k|}{|\cU'|}=  \frac{S_{\cP,k}}{|\cU'|} +  \frac{S_{\cR,k}}{|\cU'|}$, where
    \[  \frac{S_{\cP,k}}{|\cU'|} = 
    \frac{|\cP_k(U)|\cdot|\cP_k(U')|}{|\cP_k(V)|}\ \ \mbox{and}\ \  
    \frac{S_{\cR,k}}{|\cU'|} = \frac{|\cR_k(U)|\cdot|\cR_k(U')|}{|\cR_k(V)|};
    \]
    \item   $\displaystyle   \frac{S_{\cP,2}}{|\cU'|}
    =\frac{1}{12}\prod_{i=0}^{1}\frac{(1-2^{-n+2i})(1-2^{-n'+2i})}
         {1-2^{-n-n'+2i}};$  and 
  \item
    $\displaystyle
      \frac{S_{\cP,3}}{|\cU'|}<\frac{1}{112}\cdot\frac{S_{\cP,2}}{|\cU'|},\qquad
      \frac{S_{\cR,2}}{|\cU'|}=\frac{1}{12}\cdot\frac{T_1}{|\cU'|},\qquad
      \frac{S_{\cR,3}}{|\cU'|}=\frac{1}{16}\cdot\frac{S_{\cP,2}}{|\cU'|}.
    $
  \end{enumerate}
\end{lemma}
    
\begin{proof}
  (a)~The first formula holds since $S_k=S_{\cP,k}+S_{\cR,k}$ by \eqref{E:Si}.
  By Lemma~\ref{L:Sp2}(b),  $S_{\cP,k}=0$ if and only if $(n,k)=(4,3)$, 
  and in this case $\cP_k(U)=\emptyset$ and  the formula for $S_{\cP,k}/|\cU'|$ holds trivially.
  Suppose that $(n,k)\ne(4,3)$ and hence $S_{\cP,k}\ne0$. Then
  by  Lemma~\ref{L:Sp2}(b),
  $S_{\cP,k}=|\cP_k(U)|\cdot|\cU'(W)|$ for a chosen $W\in\cP_k(U)$.
  Applying Lemma~\ref{L:DC} with $\cW=\cP_k(V)$, so $W\in\cW\cap U$
  and for a chosen $U'\in\cU'$, gives
  $|\cP_k(V)|\cdot |\cU'(W)|=|\cP_k(U')|\cdot |\cU'|$.
  Thus, 
  \[
    \frac{S_{\cP,k}}{|\cU'|} = \frac{|\cP_k(U)|\cdot|\cU'(W)|}{|\cU'|}
    = \frac{|\cP_k(U)|\cdot|\cP_k(U')|}{|\cP_k(V)|}.
  \]
  A similar argument verifies the formula for $S_{\cR,k}/|\cU'|$.
  This proves part~(a).

  (b)~Setting $q=2$ in~\eqref{E:Pkq} gives
    \begin{equation}\label{E:Pk}
    |\cP_k(U)|
    =\prod_{i=0}^{k-1}\frac{2^{n-2i}-1}{2^{k-i}-1}
    =2^{k(n-k)-\binom{k}{2}}\prod_{i=0}^{k-1}\frac{1-2^{-n+2i}}{1-2^{-k+i}},
    \end{equation}
    with analogous expressions for $|\cP_k(U')|$ and $|\cP_k(V)|$. These
    expressions with
    $x\vcentcolon=({k(n-k)-\binom{k}{2})+(k(n'-k)-\binom{k}{2})-(k(n+n'-k)-\binom{k}{2}})$ yield
  \begin{align*}
  \frac{S_{\cP,k}}{|\cU'|}&=
  2^x
  \prod_{i=0}^{k-1}\frac{1{-}2^{-n+2i}}{1{-}2^{-k+i}}
  \cdot \frac{1{-}2^{-n'+2i}}{1{-}2^{-k+i}}
  \cdot \frac{1{-}2^{-k+i}}{1{-}2^{-n-n'+2i}}\\ 
  &=2^{-k(3k-1)/2}\prod_{i=0}^{k-1}\frac{(1-2^{-n+2i})(1-2^{-n'+2i})}{(1-2^{-k+i})(1-2^{-n-n'+2i})}.   
  \end{align*}
  Moreover $2^{k(3k-1)/2}\prod_{i=0}^{k-1}(1-2^{-k+i})$ equals 12 if $k=2$,
  and 1344 if $k=3$. For $k=2$, this verifies the stated formula
  for $S_{\cP,2}/|\cU'|$, proving part (b), and for $k=3$, we have
  \[
    \frac{S_{\cP,3}}{|\cU'|}
    =\frac{1}{1344}\prod_{i=0}^{2}\frac{(1-2^{-n+2i})(1-2^{-n'+2i})}
         {1-2^{-n-n'+2i}}.
  \]
  
  (c)~First we relate $S_{\cP,3}/|\cU'|$ to $S_{\cP,2}/|\cU'|$.
  \begin{align*}
    \frac{S_{\cP,3}}{|\cU'|}
    &=\frac{1}{112{\cdot}12}
    \left(\prod_{i=0}^{1}\frac{(1-2^{-n+2i})(1-2^{-n'+2i})}{1-2^{-n-n'+2i}}\right)
    \frac{(1-2^{-n+4})(1-2^{-n'+4})}{1-2^{-n-n'+4}}\\
    &< \frac{1}{112} \frac{S_{\cP,2}}{|\cU'|}.
  \end{align*}

  We next apply Lemma~\ref{L:DC} with $\cW=\cR_2(U)=\Binom{V}{2}$
  and note that the subspace $W_{\cR,2}$ of Lemma~\ref{L:Sp2}(b) lies in $\cW$.
  Then 
  $|\Binom{V}{2}|\cdot |\cU'(W_{\cR,2})|=|\Binom{U'}{2}|\cdot |\cU'|$
  by Lemma~\ref{L:DC}. Thus, by Lemma~\ref{L:Sp2}(b), 
  \[
    \frac{S_{\cR,2}}{|\cU'|}=\frac{|\Binom{U}{2}|\cdot |\cU'(W_{\cR,2})|}{|\cU'|} 
    = \frac{|\Binom{U}{2}|\cdot |\Binom{U'}{2}|}{|\Binom{V}{2}|}. 
  \]
  The number of non-degenerate $2$-subspaces of $U$ is
  \begin{equation}\label{E:nd2}
    |\cR_2(U)|=\left|\Binom{U}{2}\right|
    =\frac{|\Sp_n(2)|}{|\Sp_2(2)||\Sp_{n-2}(2)|}
    =\frac{2^{2(n-1)}(1-2^{-n})}{3},
  \end{equation}
  with analogous expressions for  $U', V$, and these yield
  \begin{align*}
    \frac{S_{\cR,2}}{|\cU'|}
    &=\frac{2^{2(n-1)}(1-2^{-n})}{3} \cdot \frac{2^{2(n'-1)}(1-2^{-n'})}{3}
      \cdot \frac{3}{2^{2(n+n'-1)}(1-2^{-n-n'})} \\
    &= \frac{(1-2^{-n})(1-2^{-n'})}{12(1-2^{-n-n'})} 
      =  \frac{1}{12}\cdot\frac{T_1}{|\cU'|},
  \end{align*}
  where the last equality follows from Lemma~\ref{L:Sp2}(a). 
    
  It remains to prove the formula for $S_{\cR,3}/|\cU'|$. Since $n\ge4$,
  both $S_{\cR,3}$ and $S_{\cP,2}$ are nonzero.
  We use Lemma~\ref{L:DC} with $G=\Sp(V)$ and
  $\cW = \cR_3(V)$ the $G$-orbit of $3$-subspaces
  of $V$ which are not totally isotropic. The subspace
  $W_{\cR,3}$ of Lemma~\ref{L:Sp2}(b) lies in $\cW$ and
  Lemma~\ref{L:DC} implies that  
  $|\cW|\cdot |\cU'(W_{\cR,3})|=|\cW(U')|\cdot |\cU'|$.
  Since $\cW(U')=\cR_3(U')$, Lemma~\ref{L:Sp2}(b) yields
  \[
  \frac{S_{\cR,3}}{|\cU'|}=
  \frac{|\cR_3(U)|\cdot |\cU'(W_{\cR,3})|}{|\cU'|} 
  =\frac{|\cR_3(U)|\cdot |\cR_3(U')|}{|\cR_3(V)|}. 
  \]
  We evaluate $|\cR_3(U)|$ as follows:
  for $W\in\cR_3(U)$, its radical
  $R=W\cap W^\perp$ is 1-dimensional and is determined by $W$, and $W/R$
  is a non-degenerate 2-subspace of  $R^\perp/R$. Thus the number of such
  subspaces $W$ contained in $U$ equals the number $2^n-1$ of choices
  for $R$ in $U$ times the number  of non-degenerate 2-subspaces of
  the non-degenerate $(n-2)$-space  $(R^\perp\cap U)/R$, which by~\eqref{E:nd2}
  equals $2^{2(n-3)}(1-2^{-n+2})/3$. Thus
  \[
  \left|\cR_3(U)\right|
  = (2^n-1) \frac{2^{2(n-3)}(1-2^{-n+2})}{3}
  = \frac{2^{3(n-2)}}{3} \prod_{i=0}^1 (1-2^{-n+2i})   
  \] 
  with analogous expressions for $\cR_3(U'), \cR_3(V)$. These yield
  \begin{align*}
  \frac{S_{\cR,3}}{|\cU'|}&=
  \frac{2^{3(n-2)}\prod_{i=0}^1 (1-2^{-n+2i})\cdot 2^{3(n'-2)}
    \prod_{i=0}^1 (1-2^{-n'+2i})}{3\cdot 2^{3(n+n'-2)}\prod_{i=0}^1 (1-2^{-n-n'+2i})}  \\
  &= \frac{1}{2^6\cdot 3}\cdot \prod_{i=0}^{1}\frac{(1-2^{-n+2i})(1-2^{-n'+2i})}
         {1-2^{-n-n'+2i}}\ =  \frac{1}{16}\cdot \frac{S_{\cP,2}}{|\cU'|},
  \end{align*}
  where the last equality follows from part (b). This completes the proof.
\end{proof}  

\begin{proof}[Proof of Theorem~$\ref{T:Sp}$]
  By Lemma~\ref{T:IE}, we have $\varphi<1/(q-1)$ for all $n,n',q$, and
  $\varphi<\frac{3}{2}q^{-1}$ if either $q\ge3$, or 
  $\min\{n,n'\}=2$. It remains to consider the case 
  $q=2$ with $\min\{n,n'\}\ge  4$. By Lemma~\ref{L:Sp2}(b),
  \[
  |\cT|\le T_1 -2S_2+28S_3 =T_1-2(S_{\cP,2}+S_{\cR,2}) + 28(S_{\cP,3} + S_{\cR,3}).
\] 
Next, applying Lemma~\ref{L:Sp4}(c) gives
\[
\varphi=  \frac{|\cT|}{|\cU'|}
  < \frac{T_1}{|\cU'|}-\frac{2\cdot S_{\cP,2}}{|\cU'|}
    - \frac{2\cdot T_1}{12\cdot|\cU'|}
    + \frac{28\cdot S_{\cP,2}}{112\cdot |\cU'|}
    + \frac{28\cdot S_{\cP,2}}{16\cdot|\cU'|}
  = \frac{5}{6}\frac{T_1}{|\cU'|}.
\]
Finally, $T_1/|\cU'|<1$  by Lemma~\ref{L:Sp2}(a), and hence
$|\cT|/|\cU'| < \frac{5}{6} =\frac{5}{3q}$. 
\end{proof}

\section{The orthogonal case}\label{S:O}

Let $V=(\F_q)^{n+n'}$ be a non-degenerate orthogonal space endowed
with a non-degenerate  quadratic form $Q\colon V\to\F_q$.
  We assume that $V$ is elliptic, hyperbolic or parabolic,
  and we denote its type by $\eps$ where $\eps=-,+,\circ$, respectively.

  Let $\beta\colon V\times V\to\F_q$ denote the corresponding polar form:
  \[
  \beta(u,v)=Q(u+v)-Q(u)-Q(v).
  \]
  The adjective ``non-degenerate'' can be applied  to $Q$ or $\beta$ and it means
  that $\textup{QRad}(V)=\{0\}$ or $\textup{BRad}(V)=\{0\}$ where
  \begin{align*}
    \textup{QRad}(V)&=\{u\in V\mid \textup{$Q(u)=0$ and $\beta(u,v)=0$ for all $v\in V$}\},\textup{ and}\\
    \textup{BRad}(V)&=\{u\in V\mid  \textup{$\beta(u,v)=0$ for all $v\in V$}\}.
  \end{align*}
  It is clear that
  $\textup{QRad}(V)\subseteq\textup{BRad}(V)$. Indeed, equality holds
  unless both $q$ is even and $\dim(V)$ is odd (see, for 
  example, \cite{KL}*{Proposition~2.5.1}).  
  We shall consider non-degenerate subspaces $U,U'$ of $V$ of
  dimensions~$n$ and $n'$, that is to say, $U$ and $U'$ endowed with the
  restricted forms $Q_{\mid U}$ and $Q_{\mid U'}$ 
  satisfy $\textup{QRad}(U)=\textup{QRad}(U')=\{0\}$.
  In this section we shall additionally assume that
  $\textup{BRad}(V)=\textup{BRad}(U)=\textup{BRad}(U')=\{0\}$.
  This means that, if $q$ is even,
  then both~$n$ and~$n'$ must be even.
  Thus if at least one of $n, n'$
  is odd then we assume that $q$ is odd. This explains the entries in
  Tables~\ref{tab1} and~\ref{T:type}.

Recall that
in Theorem~\ref{T:main1}, $\cU$ and $\cU'$ are orbits of subspaces under the isometry group of $V$.
Not all subspaces of a given  odd dimension are equivalent under isometries of $V$ and 
as discussed in Section~\ref{S:proof1}, we give the required details in Remark~\ref{R:type} below.


\begin{table}[!ht]
    \caption{The subspace type $\sigma$ and the intrinsic type $\tau$ of a
    subspace $U$ of $V$.\newline When $\dim(U)\dim(U^\perp)$ is odd, we
    identify $-1,1$ with $-,+$, respectively.}
  \centering
  \begin{tabular}{cccccl}
  \toprule
  $\dim(U)$&$\dim(U^\perp)$&$q$&$\sigma$&$\tau$&\textup{comment}\\
  \midrule
  \textup{even}&\textup{all}&\textup{all}&$\{-,+\}$&$\sigma$
    &$\sigma=\textup{isometry type of $Q_{\mid U}$}$\\
  \textup{odd}&\textup{even}&\textup{odd}&$\{-,+\}$&$\circ$
    &$\sigma=\textup{isometry type of $Q_{\mid U^\perp}$}$\\
  \textup{odd}&\textup{odd}&\textup{odd}&$\{-,+\}$&$\circ$
    &$\sigma=\delta(U)$\hskip2mm \textup{see \eqref{E:delta}}\\
  \bottomrule
  \end{tabular}\label{T:type}
\end{table}

\begin{remark}\label{R:type}
  The \emph{isometry type}  $\eps\in\{+,-,\circ\}$ where
  \[  
  +=\textup{hyperbolic},\quad -=\textup{elliptic},\quad
  \circ=\textup{parabolic}
  \]
 of a non-degenerate
 orthogonal space $V$ (or its quadratic form $Q$) is defined as usual
 in terms of the Witt index of the space, see for
 example~\cite{Taylor}*{p.~139}.  However, there are two possibilities
 for the type of a non-degenerate subspace $U$ of $V$. We call one the
 \emph{intrinsic type of $U$} (we denote it by $\tau$). It is related
 to the restricted quadratic form $Q_{\mid U}$ which makes $U$ into a
 non-degenerate space.  The other is the \emph{subspace type of $U$}
 (we denote it by $\sigma$) and it depends on how $U$ embeds into $V$.
 To be completely unambiguous, it should be called the ``type of $U$
 as a subspace of $V$'', but when the parent space~$V$ is unambiguous
 we sometimes abbreviate this to ``subspace type'' or simply
 ``type''. When the product $\dim(U)\dim(U^\perp)$ is odd, and hence
 $q$ is odd, the subspace type depends on a function $\delta(U)$ defined
 in~\eqref{E:delta}  depending on the \emph{discriminant}
 $\textup{Disc}(U)$ of $U$. (We give an explicit instance 
 of such subspaces in Example~\ref{Ex}.) The discriminant is $\square$ or $\boxtimes$,
 according as a Gram matrix for the polar form of $Q_{\mid U}$ has
 square or non-square determinant, respectively,
 see~\cite{KL}*{p.~32}. Define
 \begin{equation}\label{E:delta}
  \delta(U)=
  \begin{cases}
    \phantom{-}1&\textup{if $\textup{Disc}(U)=\square$,}\\
    -1&\textup{if $\textup{Disc}(U)=\boxtimes$.}
  \end{cases}
  \end{equation}
  \begin{enumerate}
  \item[$\circ$]  The\emph{ intrinsic type} $\tau\in\{-,\circ,+\}$ of $U$ is the isometry
    type of $Q_{\mid U}$.
  \item[$\circ$] The \emph{subspace type} $\sigma\in\{-,+\}$ of $U$ (or more precisely, the type of $U$
    as a subspace of $V$) is described in Table~\ref{T:type} and satisfies
    the following:
  \begin{enumerate}[{\rm (i)}]
  \item $\sigma$ is the isometry type of $Q_{\mid U^\perp}$ if
    $\dim(U)$ is odd and $\dim(U^\perp)$ is even i.e. if the product
    $\dim(V)\dim(U)$ is odd,
  \item $\sigma$ is the isometry type of $Q_{\mid U}$ if $\dim(U)$
    is even, and
  \item $\sigma$ is $\delta(U)$ if both $\dim(U)$ and $\dim(U^\perp)$ are odd.
  \end{enumerate}
  \end{enumerate}
  Recall that if one of $\dim(U)$ and $\dim(U^\perp)$ is odd, then we assume that
  $q$ is odd, so $\delta(U)$ is defined in~(iii).
The intrinsic type of $U$ is relevant when determining ``internal''
properties of $U$, such as the number of totally singular 1-subspaces,
while the subspace type of $U$ is relevant when considering
``embedding'' properties such as the number of subspaces of $V$
isometric to~$U$.  \emph{We use subspace type in defining $\cU$ and
$\cU'$. }
\end{remark}


In this section we prove Theorem~\ref{T:main1} in the orthogonal case by establishing the following result.

\begin{theorem}\label{T:O}
  Suppose that $n,n'\ge1$ and $V=(\F_q)^{n+n'}$ is a non-degenerate orthogonal
  space of isometry type $\eps$, such that $q\ge3$, and $q$ is odd if at least one of $n, n'$ is odd.
  Let $\cU=\Binom{V}{n}^\eps_\sigma$ and $\cU'=\Binom{V}{n'}^\eps_{\sigma'}$, where
  $\sigma$ and $\sigma'$ are subspace types as in Table~$\ref{T:type}$,
  and fix $U\in \cU$. Then either 
  \[
 \varphi= \frac{|\{U'\in \cU'\mid U\cap U'\ne0\}|}{|\cU'|}
  \le\frac{43}{16q},
  \]
  or $(n,n',q,\eps) = (1,1,3, +)$, $\sigma=\sigma'$, and $\varphi=1$.
\end{theorem}

The first result Lemma~\ref{L:DiscType} shows that, if $n$ is even and
$q$ is odd, then the discriminant $\textup{Disc}(U)$ determines the intrinsic
type of $U$, and conversely. This is not always the case when $nq$ is odd, see line 3 of Table~\ref{T:type}
where subspaces with the same intrinsic type $\circ$ may have different discriminants (and subspace types), see Example~\ref{Ex}. 

\begin{example}\label{Ex}
  Suppose $q$ is odd and $V=\langle e_1,f_1,e_2,f_2,e_3,f_3\rangle=(\F_q)^6$
  is endowed with the hyperbolic quadratic form $Q\colon V\to\F_q$ defined by
  \[
    Q(x_1e_1+y_1f_1+x_2e_2+y_2f_2+x_3e_3+y_3f_3)=x_1y_1+x_2y_2+x_3y_3.
  \]

  As $q$ is odd, we can choose $\lambda\in\F_q^\times$ such that $\lambda$
  is a non-square.  Restricting to
  $U=\langle e_1,f_1,e_3-\frac{\lambda}{2} f_3\rangle$
  and $U'=\langle e_2,f_2,e_3-\frac{1}{2}f_3\rangle$ gives
  \begin{align*}
  Q(xe_1+yf_1+z(e_3-\frac{\lambda}{2} f_3))&=xy-\frac{\lambda}{2} z^2,\quad\textup{and}\\
  Q(xe_2+yf_2+z(e_3-\frac{1}{2}f_3))&=xy-\frac{1}{2}z^2.
  \end{align*}
  Therefore $U$ and $U'$ are non-degenerate parabolic 3-subspaces.
  The restricted forms $Q_{\mid U}$ and $Q_{\mid U'}$ have polar forms $\beta_U$
  and $\beta_{U'}$. The (symmetric) bilinear forms and Gram matrices
  defining $\beta_U$ and $\beta_{U'}$ are
  \begin{align*}
    x_1y_2+x_2y_1-\lambda z_1z_2\quad\textup{and}\quad &x_1y_2+x_2y_1-z_1z_2\\
    \begin{pmatrix}
      0&1&0\\1&0&0\\0&0&-\lambda
    \end{pmatrix}\quad\textup{and}\quad&
    \begin{pmatrix}
      0&1&0\\1&0&0\\0&0&-1
    \end{pmatrix}.
  \end{align*}
  The Gram matrices have determinants $\lambda,1$ so
  $\textup{Disc}(U)=\boxtimes,\textup{Disc}(U')=\square$.
  Hence $U$ and $U'$ are non-isometric with intrinsic type $\circ$,
  and subspace types
  $\boxtimes$ and $\square$, respectively. Moreover, $V=U\oplus U'$ holds
  as $\lambda\ne1$, and $U^\perp\ne U'$.
\end{example}

\begin{lemma}\label{L:DiscType}
Suppose that $V$ is a non-degenerate orthogonal space over $\F_q$
where $q$ is odd. Let $W$ be a non-degenerate subspace of $V$, and define
\begin{enumerate}[{\rm (a)}]
  \item Then $\delta(V)=\delta(W)\delta(W^\perp)$ where $\delta$ is defined in~\eqref{E:delta}; and
  \item If $\dim(W)=2k$ is even and $W$ has intrinsic type $\tau\in\{-,+\}$,
    then $\delta(W)=\tau\cdot(-1)^{\dim(W)\nicefrac{(q-1)}{4}}=\tau\cdot(-1)^{k\nicefrac{(q-1)}{2}}$.
\end{enumerate}
\end{lemma}

\begin{proof}
  (a)~It follows from \cite{KL}*{Proposition 2.5.11(i)} that $\delta(V)=\delta(W)\delta(W^\perp)$\kern-0.5pt.

  (b)~Suppose now that $n\vcentcolon=\dim(W)$ is even and the isometry group of
  $W$ is $\textup{O}_n^\tau(q)$. Then paraphrasing~\cite{KL}*{Proposition 2.5.10, p.\;33} shows that
\begin{center}
$\delta(W)=1$ if and only if $\textup{Disc}(W) = \square$ if and only if $\tau=(-1)^{n(q-1)/4}$.
\end{center}
Hence we also have
\begin{center}
$\delta(W)=-1$ if and only if $\textup{Disc}(W) = \boxtimes$ if and only if $\tau=-(-1)^{n(q-1)/4}$.
\end{center}
Combining these cases gives $\delta(W)=\tau (-1)^{n(q-1)/4}$.
\end{proof}

Note that $\cU\vcentcolon=\Binom{V}{n}^\eps_\sigma$ and $\cU'\vcentcolon=\Binom{V}{n'}^\eps_{\sigma'}$
are orbits under the isometry group $G(V)$. \emph{Our convention
is that $U\in\Binom{V}{n}^\eps_\sigma$ has intrinsic type $\tau$
and $U'\in\Binom{V}{n'}^\eps_{\sigma'}$ has intrinsic type $\tau'$.}
If precisely one of $n$ and $n'$ is odd (and hence $q$ is odd), then we will
interchange the roles of $U$ and $U'$, if necessary, and assume
that $n'$ is odd.

\subsection{Orbits on \texorpdfstring{$1$}{}-subpaces}\label{s:1spaces}

Let $U$ have intrinsic type $\tau$, and recall from Table~\ref{T:type} 
that $\tau=\sigma$ if $n$ is even and $\tau=\circ$ if $n$ is odd.
It is convenient to identify the possible symbols $-,\circ,+$ for $\tau$
with the numbers $-1,0,1$ respectively.
It follows from~\cite{Taylor}*{Theorem~11.5} that the number of
non-zero singular vectors of $U$ is
\[
  q^{n-1}-\tau q^{n/2-1}+\tau q^{n/2}-1
  =q^{n-1}(1-\tau q^{-n/2}+\tau q^{-n/2+1}-q^{-n+1}).
\]
If $n$ is odd, then $\tau=\circ$ and the middle two terms disappear. Thus we
define a function $\gamma^\tau_n(q)$ as follows
\begin{align}\label{E:gamma}
  \gamma^\tau_n(q) &\vcentcolon=1-\tau q^{-n/2}+\tau q^{-n/2+1}-q^{-n+1}
  \quad\textup{ for $\tau\in\{-,\circ,+\}$.}
\end{align}
Dividing the number of (non-zero) singular vectors by $q-1=q(1-q^{-1})$ gives
the number of (totally) singular 1-subspaces. Thus the above equations show
\begin{align}\label{E:cP1}
  |\cP_1(U)|=\frac{q^{n-2}\gamma^\tau_n(q)}{1-q^{-1}}
  =q^{n-2}\left(\frac{1-q^{-n+1}}{1-q^{-1}}+\tau q^{-n/2+1}\right).
\end{align}
Using Equation~\eqref{E:cP1} and $|\cN_1(U)|=(q^n-1)/(q-1)-|\cP_1(U)|$ yields
\begin{equation}\label{E:NP}
  |\cN_1(U)|=q^{n-1}(1-\tau q^{-n/2})  =q^{n-1}(1-\tau q^{-\lfloor n/2\rfloor})
\end{equation}
where $\tau=0$ if $n$ is odd.

Before proceeding, we bound the
quantity $\gamma^\tau_n(q)$ in~\eqref{E:gamma}.


\begin{lemma}\label{L:gamma}
  If $\tau\in\{-,+\}$, then
  $\gamma_{n}^\tau(q)=(1+\tau q^{-n/2+1})(1-\tau q^{-n/2})$,
  and if $k\ge 1$ and $q>1$, then 
  \[
    1-q^{-k+1}< \gamma_{2k}^-(q)< \gamma_{2k+1}^\circ(q)
    <1< \gamma_{2k}^+(q) < (1-q^{-1}) \left(\frac{1}{1-q^{-1}} +  q^{-k+1}\right).
  \]
\end{lemma}

\begin{proof}
  If $\tau\in\{-,+\}$, then 
  $\gamma_{n}^\tau(q)=(1+\tau q^{-n/2+1})(1-\tau q^{-n/2})$ because we identify
  $\pm$ with $\pm1$ and then $\tau^2=1$.
  The flanking inequalities are
  \begin{align*}
    \gamma_{2k}^+(q) &= 1 + q^{-k+1} -q^{-k}  - q^{-2k+1}\\
    &<1+(1-q^{-1})q^{-k+1} = (1-q^{-1}) \left(\frac{1}{1-q^{-1}} + q^{-k+1}\right),
  \end{align*}
  and $1-q^{-k+1}<(1 - q^{-k+1})(1 +q^{-k})=\gamma_{2k}^-(q)$.
  The remaining inequalities
  $\gamma_{2k}^-(q)<\gamma_{2k+1}^\circ(q) <1< \gamma_{2k}^+(q)$
  are immediate since $\gamma_{2k+1}^\circ(q)=1-q^{-2k}$.
\end{proof}

If $q$ is even then $G(U)$ is transitive on $\cN_1(U)$. However, if $q$
is odd and $n=\dim(U)>1$, then $G(U)$ has two orbits in  $\cN_1(U)$
which we denote $\cN_1^\alpha(U)$ where $\alpha\in\{-,+\}$. The definition
of $\cN_1^\alpha(U)$ depends on $\alpha$, the parity of~$n$, and the
function $\delta$ in~\eqref{E:delta} as follows:
\begin{equation}\label{E:cN1alpha}
\cN_1^\alpha(U)=
\begin{cases}
  \{W{\in}\cN_1(U)\mid \delta(W)=\alpha\}&\textup{if $n$ is even,}\hskip4.5mm(\ref{E:cN1alpha}a)\\
  \{W\in\cN_1(U)\mid \textup{$W^\perp{\cap} U$ has type $\alpha$}\}
  &\textup{if $n{>}1$ is odd.\hskip0.5mm(\ref{E:cN1alpha}b)}
\end{cases}
\end{equation}
We comment on the word `type' in~(\ref{E:cN1alpha}b).
If $W\in\cN_1(U)$ then $W^\perp\ge U^\perp$, and since $U\cap U^\perp=0$
we see that $W^\perp\cap U$ is a
hyperplane of $U$. Therefore in (\ref{E:cN1alpha}b), where $n$ is odd, 
$\dim(W^\perp\cap U)$ is even so $\alpha$ is both
the intrinsic type and subspace type of $W^\perp\cap U$ (see line~1 of Table~\ref{T:type}).
The next result determines the subset of subspaces in $\cN_1^\alpha(V)$
which are contained in $U$.

\begin{table}[!ht]
  \caption{Values of $\cN_1^\alpha(V)\cap\cN_1(U)$ as a function of $\dim(U)$ and
    $\dim(U^\perp)$.}
  \centering
  \begin{tabular}{llcl}
  \toprule
  $\dim(U)$&$\dim(U^\perp)$&$\cN_1^\alpha(V)\cap\cN_1(U)$&\textup{definitions}\\
  \midrule
  $2m$&$2m'$&$\cN_1^{\alpha}(U)$& $m\vcentcolon=\lfloor n/2\rfloor$, $m'\vcentcolon=\lfloor n'/2\rfloor$ \\
  $2m$&$2m'+1$&$\cN_1^{\alpha\beta}(U)$
    &$\beta\vcentcolon=\delta(V)(-1)^{(m+m')(q-1)/2}$\\
  $2m+1$&$2m'$&$\cN_1^{\alpha\gamma}(U)$
    &$\gamma\vcentcolon=\delta(U^\perp)(-1)^{m'(q-1)/2}$\\
  $2m+1$&$2m'+1$&$\cN_1^{\alpha\eta}(U)$
    &$\eta\vcentcolon=\delta(U)(-1)^{m(q-1)/2}$\\
  \bottomrule
  \end{tabular}\label{T:Ant}
\end{table}

\begin{lemma}\label{L:int}
    Suppose that $q$ is odd and $U$ is a non-degenerate $n$-subspace of $V$ with
    $n>1$. For $\alpha\in\{-,+\}$, let $\cN_1^\alpha(V)$ be as in
    $\eqref{E:cN1alpha}$. Then $\cN_1^\alpha(V)\cap\cN_1(U)$
    is as in  Table~$\ref{T:Ant}$.
\end{lemma}

\begin{proof}
  If $\dim(U)=2m$ and $\dim(U^\perp)=2m'$, then it follows
  from~(\ref{E:cN1alpha}a) that $\cN_1^\alpha(V)\cap\cN_1(U)=\cN_1^\alpha(U)$.
  This proves the first line of Table~\ref{T:Ant}.
  
  Suppose next that $\dim(V)$ is odd as in lines 2 or 3 of Table~\ref{T:Ant}.
  If $W\in\cN_1^\alpha(V)$, then $\dim(W^\perp)=2(m+m')$ and $W^\perp$
  has intrinsic type $\alpha$ by~(\ref{E:cN1alpha}b). Thus,
  by Lemma~\ref{L:DiscType}(b),
  $\delta(W^\perp)=\alpha(-1)^{(m+m')(q-1)/2}$, and so
  by Lemma~\ref{L:DiscType}(a), and with $\beta$ as in Table~\ref{T:Ant},
  we have
  \begin{equation}\label{e:deltaW}
  \delta(W)=\delta(V)\delta(W^\perp)=\delta(V)\cdot \alpha(-1)^{(m+m')(q-1)/2}
    =\alpha\beta.
  \end{equation}
  We use the symbol $\oPerp$ to denote `perpendicular direct sum'.
  Since $V=U\oPerp U^\perp$ and $U = W \oPerp (W^\perp \cap U)$, it follows
  from  Lemma~\ref{L:DiscType}(a)  that $\delta(U)=\delta(V)\delta(U^\perp)$ and
  $\delta (W^\perp \cap U)=\delta (W)\delta (U).$ Then, by \eqref{e:deltaW},
  and with $\gamma$ as in Table~\ref{T:Ant},
  \begin{align}\label{e:deltaW2}
    \delta (W^\perp \cap U)&= \delta(W)\delta(V)\delta(U^\perp)
    = \alpha\beta\delta(V)\delta(U^\perp)
    \\&=\alpha(-1)^{(m+m')(q-1)/2}\delta(U^\perp)=\alpha\gamma (-1)^{m(q-1)/2}.\notag
  \end{align}
  
  Suppose that $\dim(U)=2m$ and $\dim(U^\perp)=2m'+1$. Let 
  $W\in \cN_1^\alpha(V)\cap\cN_1(U)$.
  Then $\delta(W)=\alpha\beta$ by \eqref{e:deltaW}, and so $W\in\cN_1^{\alpha\beta}(U)$  by~(\ref{E:cN1alpha}a). Hence
  $\cN_1^\alpha(V)\cap\cN_1(U)\subseteq\cN_1^{\alpha\beta}(U)$. Conversely,
  if $W\in\cN_1^{\alpha\beta}(U)$, then
  $\delta(W)=\alpha\beta$  by~(\ref{E:cN1alpha}a), and 
  by Lemma~\ref{L:DiscType}(a) and the definition of $\beta$ we have 
  $\delta(W^\perp)=\delta(V)\delta(W)=\alpha(-1)^{(m+m')(q-1)/2}$.
  It then follows from Lemma~\ref{L:DiscType}(b) that $W^\perp$ has intrinsic
  type $\alpha$. Hence, by~(\ref{E:cN1alpha}b), 
  $W\in \cN_1^\alpha(V)\cap\cN_1(U)$ and the reverse
  inclusion  $\cN_1^{\alpha\beta}(U)\subseteq\cN_1^\alpha(V)\cap\cN_1(U)$ holds.
  This proves line 2 of Table~\ref{T:Ant}.
  
  Suppose next that $\dim(U)=2m+1$ and $\dim(U^\perp)=2m'$, and let 
  $W\in \cN_1^\alpha(V)\cap\cN_1(U)$.  Then $\delta(W^\perp\cap U)=\alpha\gamma (-1)^{m(q-1)/2}$ by \eqref{e:deltaW2}. Now $\dim(W^\perp\cap U)=2m$, 
  and so by  Lemma~\ref{L:DiscType}(b), $W^\perp\cap U$ has intrinsic 
  type $\alpha\gamma$, and hence $W\in\cN_1^{\alpha\gamma}(U)$  by~(\ref{E:cN1alpha}b).
  Hence
  $\cN_1^\alpha(V)\cap\cN_1(U)\subseteq\cN_1^{\alpha\gamma}(U)$. Conversely,
  if $W\in\cN_1^{\alpha\gamma}(U)$, then $W^\perp\cap U$ has intrinsic type 
  $\alpha\gamma$ by~(\ref{E:cN1alpha}b), and hence  $\delta(W^\perp\cap U)=\alpha\gamma(-1)^{m(q-1)/2}$ by Lemma~\ref{L:DiscType}(b). Applying
  Lemma~\ref{L:DiscType}(a) three times gives 
  \[
  \delta(W^\perp)=\delta(V)\delta(W) =\delta(V) \delta(U)\delta(W^\perp\cap U)
  =\delta(U^\perp)\delta(W^\perp\cap U).
  \]
  Then, since $\delta(W^\perp\cap U)=\alpha\gamma(-1)^{m(q-1)/2}$, 
  and using the definition of $\gamma$, we conclude that $\delta(W^\perp)=
  \alpha(-1)^{(m+m')(q-1)/2}$, and hence that 
  $W^\perp$ has intrinsic type $\alpha$ by Lemma~\ref{L:DiscType}(b).
  Thus $W\in \cN_1^\alpha(V)$ by~(\ref{E:cN1alpha}b) and the reverse inclusion
  $\cN_1^{\alpha\gamma}(U)\subseteq\cN_1^\alpha(V)\cap\cN_1(U)$ holds.
  This proves line 3 of Table~\ref{T:Ant}.

  Finally, suppose that $\dim(U)=2m+1$ and $\dim(U^\perp)=2m'+1$, and let
  $W\in\cN_1^\alpha(V)\cap\cN_1(U)$. Then $\dim(V)$ is even
  and $\delta(W)=\alpha$ by~(\ref{E:cN1alpha}a). Furthermore,
  $U=W\oPerp(W^\perp\cap U)$ so, by  Lemma~\ref{L:DiscType}(a), and with 
  $\eta$ as in Table~\ref{T:Ant}, we have
  \[
  \delta(W^\perp\cap U)=\delta(W)\delta(U)=\alpha\delta(U)=\alpha\eta(-1)^{m(q-1)/2}.
  \]
  Thus $W^\perp\cap U$ has intrinsic type $\alpha\eta$
  by Lemma~\ref{L:DiscType}(b), and hence 
  $W\in\cN_1^{\alpha\eta}(U)$ by~(\ref{E:cN1alpha}b) and
  $\cN_1^\alpha(V)\cap\cN_1(U)\subseteq\cN_1^{\alpha\eta}(U)$ holds.
  Conversely, if $W\in\cN_1^{\alpha\eta}(U)$, then $W^\perp\cap U$ has intrinsic type
  $\alpha\eta$ by~(\ref{E:cN1alpha}b), and so  
  $\delta(W^\perp\cap U)=\alpha\eta(-1)^{m(q-1)/2}$
  by Lemma~\ref{L:DiscType}(b). Hence
  $\delta(W)=\delta(U)\delta(W^\perp\cap U)=\alpha$ and
  $\cN_1^{\alpha\eta}(U)\subseteq\cN_1^\alpha(V)\cap\cN_1(U)$ holds.
  This proves line 4 of Table~\ref{T:Ant}.
\end{proof}

Now we obtain a formula for $|\cN_1^\alpha(U)|$, for $\alpha\in\{+,-\}$.
If $n=2m$ is even, then 
$|\Binom{U}{1}^\sigma_\boxtimes|=|\Binom{U}{1}^\sigma_\square|$, and so
$|\cN_1^\pm(U)|=\frac12|\cN_1(U)|$. Thus, by \eqref{E:NP}, 
$|\cN^\alpha_1(U)|=q^{n-1}(1-\tau q^{-m})/2$
by~\eqref{E:NP} where $\tau$ is the intrinsic type of $U$.
If $\dim(U)=2m+1$ is odd, then by~\eqref{E:cN1alpha} we have
\[
  |\cN_1^\alpha(U)|=\left|\Binom{U}{n-1}^\circ_\alpha\right|
  =\frac{|\GO_{2m+1}^\circ(q)|}{|\GO_{1}^\circ(q)\times\GO_{2m}^\alpha(q)|}
  =\frac{q^{2m}(1+\alpha q^{-m})}{2}.
\]
In summary,
\begin{equation}\label{E:cNalpha}
  |\cN_1^\alpha(U)|=
  \begin{cases}
    \displaystyle\frac{q^{n-1}(1-\tau q^{-m})}{2}&\textup{if $n=2m$ and $U$ has type $\tau$,\hskip3mm (\ref{E:cNalpha}\textup{a})}\\[4mm]
    \displaystyle \frac{q^{n-1}(1+\alpha q^{-m})}{2}&\textup{if $n=2m+1$.}
    \hskip29mm (\ref{E:cNalpha}\textup{b})
  \end{cases}
\end{equation}
Note that $\tau$ in (\ref{E:cNalpha}\textup{a}) is both the subspace type and the intrinsic type of $U$.

\subsection{Proof of Theorem~\texorpdfstring{$\ref{T:O}$}{} for \texorpdfstring{$\min\{n,n'\}=1$}{}}\label{s:onedim}

We have sufficient information to deal with the case of Theorem~\ref{T:O} 
where one of $n$ or $n'$ is equal to $1$. Here we come across the 
exceptional case where the general bound on the parameter $\varphi$ in \eqref{E:rho} fails to hold. 

\begin{lemma}\label{n=1}
  Suppose that the hypotheses of Theorem~$\ref{T:O}$ hold with $q$ odd and 
  $\min\{ n, n'\}=1$.  Then either  
  $\varphi\le \frac{5}{2q}$, or  
  $(n,n',q,\eps)\ne(1,1,3,+)$, $\sigma=\sigma'$,  and $\varphi=1$.
  Moreover, Theorem~$\ref{T:O}$ holds if $\min\{ n, n'\}=1$.
\end{lemma}

\begin{proof}
  Since $5/2 <43/16$ it is sufficient to prove the first assertion.
  Suppose first that $n=n'=1$. If $U\in\cU$, then
  \[
  |U^{\GO_2(q)}|=\frac{|\GO_2(q)|}{|\GO_1(q)\times\GO_1(q)|}
  =\frac{2(q-\eps\cdot 1)}{2\cdot2}=\frac{q-\eps\cdot 1}{2}.
  \]
  Thus $|\cU|=|\cU'|=(q -\eps\cdot 1)/2$. If $\sigma\ne \sigma'$, then
  $\cU, \cU'$ are disjoint sets and hence $\varphi=0 < \frac{5}{2q}$.
  So suppose that $\sigma=\sigma'$. Then $U\in\cU$ and $U'\in\cU'$
  satisfy $U\cap U'\ne0$ precisely when $U=U'$, so 
  \[
  \varphi=\frac{1}{|\cU'|} = \frac{2}{q-\eps\cdot 1}\le \frac{2}{q-1}
  =\frac{2}{q}\left( 1+\frac{1}{q-1}\right).
  \]
  If $q\ge 5$ then $\varphi\le (2/q)\cdot (5/4)=5/(2q)$, as required,
  so assume that $q=3$. 
  If $\eps=-$ then $\varphi = 2/4 < 5/(2q)$. This leaves the exceptional
  case where $\eps=+$ and $q=3$. Here the space $V$ has exactly
  two non-degenerate $1$-subspaces, one of each type. Since $\sigma=\sigma'$,
  we see that $\cU=\cU'=\{U\}=\{U'\}$, and $\varphi=1$. 
  
  Suppose now that $n+n'>2$. By symmetry, we can (and will) assume
  that $n'=1$, so $\cU'=\Binom{V}{1}^\eps_{\sigma'}$ and $n\ge2$. Thus
  $\cU'=\cN_1^\alpha(V)$ as in \eqref{E:cN1alpha} for some $\alpha\in\{+,-\}$. 
  Also $n'=2m'+1$ with $m'=0$, and $n=2m$ or $2m+1$ with $m\ge 1$. Choose
  $U\in\cU=\Binom{V}{n}^\eps_\sigma$.  Then by \eqref{E:rho},
  \[
  \varphi\cdot |\cN_1^\alpha(V)|=|\{U'\in\cN_1^\alpha(V)\mid U'\cap U\ne 0\}|
  = |\cN_1^\alpha(V)\cap\cN_1(U)|.
  \]
  By Lemma~\ref{L:int}, 
  \[
  \cN_1^\alpha(V)\cap\cN_1(U)=
  \begin{cases}
   \cN_1^{\alpha\beta}(U)&\textup{if $n=2m$, where $\beta=\delta(V)(-1)^{m(q-1)/2}$,}\\
  \cN_1^{\alpha\eta}(U)&\textup{if $n=2m+1$, where $\eta=\delta(U)(-1)^{m(q-1)/2}$.}
  \end{cases}
  \]
  If $n=2m$ then the intrinsic type of $U$ is equal to $\sigma$ by
  Table~\ref{T:type}, so by \eqref{E:cNalpha},
\[
\varphi = \frac{|\cN_1^{\alpha\beta}(U)|}{|\cN_1^\alpha(V)|}
	= \frac{q^{2m-1}(1-\sigma q^{-m})/2}{q^{2m}(1+\alpha q^{-m})/2}
	\le \frac{1}{q}\cdot\frac{1+ q^{-m}}{1- q^{-m}}\le \frac{1}{q}\cdot\frac{1+ 3^{-1}}{1- 3^{-1}}=\frac{2}{q}.
\] 
Similarly, if  $n=2m+1$ then the intrinsic type of $V$ is $\eps\in\{-,+\}$,
so by~\eqref{E:cNalpha},
\[
\varphi = \frac{|\cN_1^{\alpha\eta}(U)|}{|\cN_1^\alpha(V)|}
	= \frac{q^{2m}(1+\alpha q^{-m})/2}{q^{2m+1}(1-\eps q^{-m})/2}
	\le \frac{1}{q}\cdot\frac{1+ q^{-m}}{1- q^{-m}}\le \frac{2}{q}.
\]     
Thus in both cases $\varphi\le 2/q<5/(2q)$. This completes the proof.
\end{proof}

\subsection{Strategy for the proof of Theorem~\texorpdfstring{$\ref{T:O}$}{} for \texorpdfstring{$\min\{n,n'\}\ge 2$}{}}\label{s:strategyo}

As in the hypothesis of Theorem~\ref{T:O}, $\cU=\Binom{V}{n}^\eps_\sigma$ and 
$\cU'=\Binom{V}{n'}^\eps_{\sigma'}$, and we choose $U\in\cU$ and $U'\in\cU'$.
By Lemma~\ref{n=1}, we may (and will) assume that $\min\{n,n'\}\ge 2$.
For this general case we follow the strategy described in
Section~\ref{S:strategy} to obtain an upper bound for $\varphi$
as in \eqref{E:rho} of the form given in \eqref{E:rhoUO} in terms of
quantities $c_1, c_2$ given in~\eqref{E:rhoUO-P} and~\eqref{E:rhoO-N}.   

In order to estimate $c_1$ we determine $|\cU'(W)|$ when
$W$ is a nonsingular $1$-subspace of $U$, that is $W \in \cN_1(U)$,
and to estimate $c_2$ we determine $|\cU'(W)|$ when $W$ is a singular
$1$-subspace of $U$, that is $W \in \cP_1(U)$. Here we use the notation
$\cU'(W), \cW(U),$ etc from Section~\ref{S:double}. The following lemma is key.

\begin{lemma}\label{L:key}
  Let $U\in \cU$ and let $m=\lfloor n/2\rfloor$, $m'=\lfloor n'/2\rfloor$.
  For each set $\cW(U)$ in row~$2$ of Table~$\ref{T:X}$,
  row~$3$ lists the quantity $|\cU'(W)|/|\cU'|$, where 
  $(W,U')\in \cW(U)\times \cU'$, depending on the parities of $n, n'$ and $q$.
\end{lemma}

\begin{table}[!ht]
  \caption{Possibilities for $\cW(U)$ and $|\cU'(W)|/|\cU'|$, where
    $\alpha\in\{-,+\}$, $\beta, \eta$ as in Table~\ref{T:Ant} and
    $\gamma'=\delta((U')^\perp)(-1)^{m(q-1)/2}$,
      $\eta'=\delta(U')(-1)^{m'(q-1)/2}$.}
  \centering
  \begin{tabular}{lcccc}
  \toprule
  $(n,n',q)$&{\rm\small(all, all, all)}&{\rm\small(even, even, all)}&{\rm\small(even, odd, odd)}&{\rm\small(odd, odd, odd)}\\
  \midrule
  $\cW(U)$&$\cP_1(U)$&$\cN_1(U)$&$\cN_1^{\alpha\beta}(U)$&$\cN_1^{\alpha\eta}(U)$\\
    $\displaystyle\frac{|\cU'(W)|}{|\cU'|}$&$\displaystyle\frac{|\cP_1(U')|}{|\cP_1(V)|}$&$\displaystyle\frac{|\cN_1(U')|}{|\cN_1(V)|}$&$\displaystyle\frac{|\cN_1^{\alpha\gamma'}(U')|}{|\cN_1^\alpha(V)|}$&$\displaystyle\frac{|\cN_1^{\alpha\eta'}(U')|}{|\cN_1^\alpha(V)|}$\\
  \bottomrule
  \end{tabular}\label{T:X}
\end{table}

\begin{proof}
  Let $G(V)$ and $\CO(V)$ denote the isometry group, and the conformal
  orthogonal group of $V$, respectively.    
  We use Lemma~\ref{L:DC} for the $G$-orbits $\cU=\Binom{V}{n}^\eps_\sigma$
  and $\cW$ for the various
  choices of $G$, and $\cW$ described in the first two rows of
  Table~\ref{T:XX}. 
  In each of the four cases, the elements of $\cW$ are 1-subspaces of $V$
  and the set $\cW(U)$ is non-empty.
  For $W\in \cW(U)$, the set $\cU'(W)$ defined before Lemma~\ref{L:DC} satisfies
  \[
  \cU'(W)=\{U'\in \cU'\mid W\le U'\}=\{U'\in \cU'\mid W\le U'\cap U\},
  \]
  and by Lemma~\ref{L:DC},
  \begin{equation}\label{E:key}
    |\cW|\cdot |\cU'(W)|=|\cW(U')|\cdot |\cU'|
    \qquad\textup{for all $(W,U')\in \cW(U)\times \cU'$}.
  \end{equation}
  We next justify  the values of $\cW(U)$ and $\cW(U')$ in  rows~3  and~4  of
  Table~\ref{T:XX}.

  \begin{table}\caption{Values of $\cW(U), \cW(U')$ depending on
      $G, \cW$ where 
      $\beta, \gamma', \eta, \eta'$ are as in Table~$\ref{T:X}$.}
  \centering
  \begin{tabular}{lcccc}
    \toprule
    &{\rm Case 1}&{\rm Case 2}&{\rm Case 3}&{\rm Case 4}\\
    \midrule
    $(n,n',q)$&{\rm\small(all, all, all)}&{\rm\small(even, even, all)}&{\rm\small(even, odd, odd)}&{\rm\small(odd, odd, odd)}\\
    \midrule
    $G$&$G(V)$&$\CO(V)$&$G(V)$&$G(V)$\\
    $\cW$&$\cP_1(V)$&$\cN_1(V)$&$\cN_1^\alpha(V)$&$\cN_1^{\alpha}(V)$\\
    $\cW(U)$&$\cP_1(U)$&$\cN_1(U)$&$\cN_1^{\alpha\beta}(U)$&$\cN_1^{\alpha\eta}(U)$\\
    $\cW(U')$&$\cP_1(U')$&$\cN_1(U')$&$\cN_1^{\alpha\gamma'}(U')$&$\cN_1^{\alpha\eta'}(U')$\\
    \bottomrule
  \end{tabular}\label{T:XX}
  \end{table}

  {\sc Cases 1,2.}  In Case~1 of Table~\ref{T:XX}, $\cW=\cP_1(V)$
  and the parities  of $(n,n',q)$ are  arbitrary. Further,
  $\cW(U)$ equals $\{W\in \cP_1(V)\mid W\le U\}=\cP_1(U)$, and
  $\cW(U')=\cP_1(U')$. In Case~2 of Table~\ref{T:XX}, $\cW=\cN_1(V)$
  and similar reasoning shows that $\cW(U)=\cN_1(U)$ and $\cW(U')=\cN_1(U')$.
  Here $\cW$ is a $\CO(V)$-orbit, and not a $\GO(V)$-orbit, so our use of
  Lemma~\ref{L:DC} is valid.

  {\sc Case 3.} Here $n=\dim(U)$ is  even, and both $n'=\dim(U')$ and $q$
  are odd. Now $\cW=\cN^\alpha_1(V)$ is a $G(V)$-orbit.
  The definitions of $\cW(U)$
  and $\cW(U')$ show that  $\cW(U)=\cN^\alpha_1(V)\cap\cN_1(U)$ and
  $\cW(U')=\cN^\alpha_1(V)\cap\cN_1(U')$. Thus, by Lemma~\ref{L:int}
  and lines 2 and 3 of Table~\ref{T:Ant},
  the entries for $\cW(U)$ and $\cW(U')$ in Table~\ref{T:XX} are valid.
  
  {\sc Case 4.} In this case $n=2m+1$, $n'=2m'+1$ and $q$ are odd,
  and $\cW=\cN_1^\alpha(V)$ is a $G(V)$-orbit.
  By definition, $\cW(U)=\cN_1^\alpha(V)\cap\cN_1(U)$ and
  $\cW(U')=\cN_1^\alpha(V)\cap\cN_1(U')$. By Lemma~\ref{L:int}
  and line 4 of Table~\ref{T:Ant},
  the entries for $\cW(U)$ and $\cW(U')$ in Table~\ref{T:XX} are correct.

  Thus all entries of Table~\ref{T:XX} are valid. Substituting the
  values for $\cU'(W)$ and $\cW(U')$ from Table~\ref{T:XX} into
  Equation~\eqref{E:key} implies that
  $|\cU'(W)|/|\cU'|=|\cW(U')|/|\cW|$ holds
  for all $(W,U')\in \cW(U)\times \cU'$ as claimed.
  This justifies all the entries of Table~\ref{T:X} completing the proof.
\end{proof}

\goodbreak
\subsection{Singular \texorpdfstring{$1$}{}-subspaces and a formula for \texorpdfstring{$c_2$}{}}\quad
In Lemma~\ref{L:cP1} we establish the inequality in Equation~\eqref{E:rhoUO-P} involving the quantity $c_2$ and derive a formula for $c_2$. 

\begin{lemma}\label{L:cP1}
  Suppose that $n,n'\ge2$, and let $U\in \cU$ and $U'\in\cU'$ have
  intrinsic types $\tau$ and $\tau'$, respectively.
  Then
  \[
  \frac{\left|\bigcup_{W\in\cP_1(U)} \cU'(W)\right|}{|\cU'|}\le
  \frac{|\cP_1(U)||\cP_1(U')|}{|\cP_1(V)|}
  =\frac{c_2}{q^2}\quad
  \textup{where }c_2\vcentcolon=\frac{\gamma_n^\tau(q)\gamma_{n'}^{\tau'}(q)}
                          {(1-q^{-1})\gamma_{n+n'}^\eps(q)},
  \]
  and $\gamma_n^\tau(q)$ is defined by~\eqref{E:gamma}. Moreover,
\[
  \frac{c_2}{q} 
     \le X_2\vcentcolon= \frac{(1-q^{-1})
       (\frac{1}{1-q^{-1}}  + q^{-\lfloor\frac{n}{2}\rfloor+1})
 (\frac{1}{1-q^{-1}}     + q^{-\lfloor\frac{n'}{2}\rfloor+1})}
     {q (1-q^{-\lfloor \frac{n+n'}{2}\rfloor+1}) }.
     \]
\end{lemma}

Note that $\tau=\sigma$ if $n$ is even and $\tau=\circ$ if $n$ is odd. Similarly,  $\tau'=\sigma'$ if $n'$ is even and $\tau'=\circ$ if $n'$ is odd.

\begin{proof}
  Suppose that $G=G(V)$.
  The stabiliser $G_U$ of $U$ induces $G(U)$ on $U$ and so
  acts transitively on the set
  $\{\cU'(W)\mid W\in\cP_1(U)\}$. Thus $|\cU'(W)|$ 
  is independent of the choice of $W\in \cP_1(U)$. Hence
  \begin{equation}\label{E:ZZ}
    \left|\bigcup_{W\in\cP_1(U)} \cU'(W)\right|\le |\cP_1(U)||\cU'(W)|
    \qquad\textup{for a chosen $W\in\cP_1(U)$}.
  \end{equation}
  It follows from Lemma~\ref{L:key} that
  $|\cU'(W)|/|\cU'|=|\cP_1(U')/|\cP_1(V)|$  for all $W\in\cP_1(U)$.
  The first assertion now follows from equation~\eqref{E:cP1} since
  \[
    |\cP_1(U)|=\frac{q^{n-2}\gamma_n^\tau(q)}{1-q^{-1}},\; |\cP_1(U')|
    =\frac{q^{n'-2}\gamma_{n'}^{\tau'}(q)}{1-q^{-1}}
    \;\textup{and}\;
    |\cP_1(V)|=\frac{q^{n+n'-2}\gamma_{n+n'}^\eps(q)}{1-q^{-1}}\kern-0.2pt.
  \]
  Finally, note that, using Lemma~\ref{L:gamma}, and setting
  $m=\lfloor n/2\rfloor$ and $m'=\lfloor n'/2\rfloor$,

  \[
    \frac{c_2}{q} \vcentcolon=\frac{q|\cP_1(U)||\cP_1(U')|}{|\cP_1(V)|}
    = \frac{\gamma_n^\tau(q)\gamma_{n'}^{\tau'}(q)}
           {q(1-q^{-1})\gamma_{n+n'}^\eps(q)} 
    \le \frac{\gamma_{2m}^+(q)\gamma_{2m'}^{+}(q)}
             {q(1-q^{-1})\gamma_{2\lfloor\frac{n+n'}{2}\rfloor}^-(q)}, 
  \]
  and now the upper bound for ${c_2}/{q}$ follows immediately from
  Lemma~\ref{L:gamma}.
\end{proof}

\goodbreak
\subsection{Nonsingular \texorpdfstring{$1$}{}-subspaces and the quantity \texorpdfstring{$c_1$}{}}\quad

Here we tackle the second term in the upper bound for $\varphi$ in \eqref{E:rho2};
we write this term as $c_1/q$, and obtain an explicit upper bound in terms of $n, n', q$
which holds for all isometry types for $V$ and all subspaces types of $U$ and $U'$.
Our aim in this subsection is to prove the following lemma. 

\begin{lemma}\label{L:cN1}
  Suppose that $n,n'\ge 2$, and $\cU\vcentcolon=\Binom{V}{n}^\eps_\sigma$, $\cU'\vcentcolon=\Binom{V}{n'}^\eps_{\sigma'}$, and that $U\in\cU$. If $m=\lfloor n/2\rfloor$
  and $m'=\lfloor n'/2\rfloor$, then
  \[
  \frac{\left|\bigcup_{W\in\cN_1(U)} \cU'(W)\right|}{|\cU'|}
  = \frac{c_1}{q}\quad \mbox{where}\quad c_1
  \le X_1 \vcentcolon= \frac{(1+q^{-m})(1+ q^{-m'})}{1-q^{-\lfloor (n+n')/2\rfloor}}.
  \]
  Moreover, if $n=2m$ and $n'=2m'+1$, then $c_1\le Y_1$, and if
  $n=2m+1$  and $n'=2m'+1$, then $c_1\le Z_1$, where
  \[
    Y_1\vcentcolon=\frac{(1+q^{-m})(1+q^{-m-2m'})}{1-q^{-2m-2m'}}\qquad\mbox{and}\qquad
    Z_1\vcentcolon=\frac{1+q^{-m-m'}}{1-q^{-m-m'-1}}.
  \]
\end{lemma}


\begin{proof}
Write $c_1/q\vcentcolon=|\cup_{W\in\cN_1(U)} \cU'(W)|/|\cU'|$.  We
divide the proof into three cases depending on the parities of $n$ and
$n'$. We write $n=2m$ or $2m+1$, and $n'=2m'$ or $2m'+1$, depending on
their parities, where $m, m'$ are positive integers. As mentioned
above, if precisely one of $n, n'$ is odd (and hence also $q$ is odd)
then we may assume that $n'$ is odd. Thus our cases are: (i) $n, n'$
both even; (ii) $n$ even and $n', q$ odd; and (iii) $n, n', q$ all
odd.

\medskip\noindent
\emph{Case (i): $n=2m, n'=2m'$. We prove that $c_1\le X_1$.}

\medskip\noindent
\emph{Proof of Case (i):}  
  Here the conformal orthogonal group $\CO(V)$ acts transitively 
  on $\cN_1(V)$, and the stabiliser $\CO(V)_U$ induces a transitive group
  on $\cN_1(U)$. Hence $|\cup_{W\in\cN_1(U)} \cU'(W)|\le |\cN_1(U)|\cdot|\cU'(W)|$
  for a chosen $W\in\cN_1(U)$. Moreover, by Lemma~\ref{L:key},
  $|\cU'(W)|/|\cU'|=|\cN_1(U')|/|\cN_1(V)|$, and hence
  $c_1/q\le  |\cN_1(U)|\cdot|\cN_1(U')|/|\cN_1(V)|$. 
  Using~\eqref{E:NP} three times gives
  \begin{align*}
   \frac{c_1}{q} &\le \frac{q^{n-1}(1-\sigma q^{-m})q^{n'-1}(1-\sigma' q^{-m'})}
          {q^{n+n'-1}(1-\eps q^{-m+m'})}\\
    &=\frac{q^{-1}(1-\sigma q^{-m})(1-\sigma' q^{-m'})}
          {1-\eps q^{-m+m'}}\le \frac{(1+ q^{-m})(1+ q^{-m'})}
          {q(1- q^{-m+m'})}=\frac{X_1}{q}. 
  \end{align*}

\medskip
Since the approaches for Cases (ii) and (iii) are similar, we begin by
considering them together. In both cases, $n'=2m'+1$ and $q$ is odd, and
for  $G=G(V)$, the stabiliser $G_U$ induces $G(U)$ on $U$ (by Witt's Theorem),
and hence $G_U$ has two orbits $\cN_1^\xi(U)$ on   $\cN_1(U)$, where
$\xi\in\{-,+\}$. For each~$\xi$,   the cardinality of $\cU'(W)$ is
independent of the choice of $W\in\cN_1^{\,\xi}(U)$, and we choose
$W_\xi\in\cN_1^{\,\xi}(U)$.  Then
  \begin{equation}\label{E:c1}
  \frac{c_1}{q}  \le\sum_{\xi\in\{-,+\}}
    \frac{\left|\bigcup_{W\in\cN_1^{\,\xi}(U)} \cU'(W)\right|}{|\cU'|}
  \le\sum_{\xi\in\{-,+\}}\frac{|\cN_1^{\,\xi}(U)|\cdot |\cU'(W_\xi)|}{|\cU'|}.  
  \end{equation}
For each $\xi$, $|\cN^{\,\xi}_1(U)|$ is given by~\eqref{E:cNalpha},
and we use Lemma~\ref{L:key} to determine $|\cU'(W_{\xi})|/|\cU'|$,
but for clarity we now continue with Cases (ii) and (iii) separately. 

\medskip\noindent
\emph{Case (ii): $n=2m, n'=2m'+1$, $q$ odd. We prove that $c_1\le Y_1\le X_1$.}

\medskip\noindent
\emph{Proof of Case (ii):}  
  Now $U$ has intrinsic type $\sigma$ by Table~\ref{T:type}, and so 
  $|\cN^{\,\xi}_1(U)|=q^{n-1}(1-\sigma q^{-m})/2$, by (\ref{E:cNalpha}a), for each $\xi$.
  Let $\beta=\delta(V)(-1)^{(m+m')(q-1)/2}$, with $\delta$ as in \eqref{E:delta}, 
  and write $\xi=\alpha\beta$, so $\alpha\in\{-,+\}$. Then it follows from 
  Lemma~\ref{L:key}   that  
  $|\cU'(W_{\alpha\beta})|/|\cU'|=|\cN_1^{\alpha\gamma'}(U')|/|\cN_1^\alpha(V)|$,
  with $\gamma'$ as in the caption for Table~\ref{T:X}.
  Hence by \eqref{E:c1}, and using (\ref{E:cNalpha}a) and
  (\ref{E:cNalpha}b), $c_1/q$ is at most
  \begin{align*}
    \sum_{\alpha\in\{-,+\}}
    \frac{|\cN_1^{\alpha\beta}(U)|\cdot |\cN_1^{\alpha\gamma'}(U')|}{|\cN_1^\alpha(V)|}
    &=\sum_{\alpha\in\{-,+\}}
    \frac{q^{n-1}(1-\sigma q^{-m})q^{n'-1}(1+\alpha\gamma' q^{-m'})}
         {2q^{n+n'-1}(1+\alpha q^{-m-m'})}\\
    &=\frac{q^{-1}(1-\sigma q^{-m})}{2}\sum_{\alpha\in\{-,+\}}
         \frac{1+\alpha\gamma' q^{-m'}}{1+\alpha q^{-m-m'}}\\
    &=\frac{q^{-1}(1-\sigma q^{-m})(1-\gamma' q^{-m-2m'})}{1-q^{-2m-2m'}}\\
    &\le \frac{(1+ q^{-m})(1+ q^{-m-2m'})}{q(1-q^{-2m-2m'})}\\
    &\le \frac{(1+ q^{-m})
    (1+ q^{-m'})}{q(1-q^{-m-m'})},
  \end{align*}
  where the last inequality is equivalent to
  $1+q^{-m-2m'}\le (1+q^{-m'})(1+q^{-m-m'})$, and this is clearly true for
  all positive integers $m, m'$. Hence $c_1\le Y_1\le X_1$.

\medskip\noindent
\emph{Case (iii): $n=2m+1, n'=2m'+1$, $q$ odd. We prove that
$c_1\le Z_1\le X_1$.}

\medskip\noindent
\emph{Proof of Case (iii):}
  For each $\xi\in\{-,+\}$, we have $|\cN^{\,\xi}_1(U)|=q^{n-1}(1+\xi
  q^{-m})/2$, by (\ref{E:cNalpha}b).  Let $\eta, \eta'$ be as in the
  caption for Table~\ref{T:X}, and let $\xi=\alpha\eta$, so
  $\alpha\in\{-,+\}$. Then it follows from Lemma~\ref{L:key} that
  $|\cU'(W_{\alpha\eta})|/|\cU'|=
  |\cN_1^{\alpha\eta'}(U')|/|\cN_1^\alpha(V)|$.  Hence, by
  \eqref{E:c1}, and using (\ref{E:cNalpha}a) for $|\cN_1^\alpha(V)|$
  and (\ref{E:cNalpha}b) for $|\cN_1^{\alpha\eta'}(U')|$,
   \begin{align*}
    c_1/q &\le \sum_{\alpha\in\{-,+\}}
    \frac{|\cN_1^{\alpha\eta}(U)|\cdot |\cN_1^{\alpha\eta'}(U')|}{|\cN_1^\alpha(V)|}\\
    &=\kern-10pt\sum_{\alpha\in\{-,+\}}
    \frac{q^{n-1}(1+\alpha\eta q^{-m})q^{n'-1}(1+\alpha\eta' q^{-m'})}
         {2q^{n+n'-1}(1-\eps q^{-m-m'-1})}\\
    &=\frac{1}{2q(1-\eps q^{-m-m'-1})}\sum_{\alpha\in\{-,+\}}
         (1+\alpha\eta q^{-m})(1+\alpha\eta' q^{-m'})\\
    &=\frac{(1+\eta\eta'q^{-m-m'})}{q(1-\eps q^{-m-m'-1})}
     \le \frac{(1+ q^{-m-m'})}{q(1-q^{-m-m'-1})}
     \le \frac{(1+ q^{-m})(1+ q^{-m'})}{q(1-q^{-m-m'-1})},
  \end{align*}
  where the last inequality holds because
  $1+q^{-m-m'}\le (1+q^{-m})(1+q^{-m'})$, for all positive integers $m, m'$.
  Hence $c_1\le Z_1\le X_1$.
\end{proof}

\subsection{Proof of Theorem~\texorpdfstring{$\ref{T:O}$}{}}


We use the following lemma, which is easy to verify.
\begin{lemma}\label{L:boundprod}
  Let $q, x, y,\ell\in \Z$  satisfy $q\ge 2$, $0 \le x \le y$
  and $0 \le \ell \le x$. If $a$ is a positive real number, then
  \[
  (a + q^{-x})(a+q^{-y}) \le (a + q^{-x+\ell})(a+q^{-y-\ell}).
  \]
\end{lemma}

\begin{proof} Clearly, the result holds for $\ell=0.$ Suppose that $\ell\ge 1$.
  As $a>0$, the claimed inequality is equivalent to 
  $q^{-x}+q^{-y} \le  q^{-x+\ell}+q^{-y-\ell}$ which is equivalent
  to  $q^y + q^x \le q^{y+\ell} + q^{x-\ell}$.
  This is true when $x=y$ since $2\le q^{\ell} + q^{-\ell}$ holds for $q\ge 2$.
  It also holds for $x < y$, since $q^y + q^x <q^{y+1}< q^{y+\ell}+q^{x-\ell}$.
\end{proof}



\begin{proof}[Proof of Theorem~$\ref{T:O}$]  
  If $n$ or $n'$ equals 1, then $q$ is odd and the result follows from
  Lemma~\ref{n=1}. Suppose now that $n,n'\ge2$.
  Let $m=\lfloor\frac{n}{2}\rfloor$ and $m'=\lfloor\frac{n'}{2}\rfloor$, so that
  $m, m'\ge 1$ and recall that $q\ge3$ by assumption.
  Then Equation~\eqref{E:rho2}, together with~Lemmas~\ref{L:cP1} and~\ref{L:cN1}, 
  shows that $\varphi\le (X_1+X_2)/q$ where
  \begin{align*}
    X_1 &= \frac{(1+q^{-m})(1+ q^{-m'})}{1-q^{-\lfloor (n+n')/2\rfloor}},\\  
    X_2 &= \frac{(1-q^{-1})(\frac{1}{1-q^{-1}}  + q^{-m+1})
            (\frac{1}{1-q^{-1}}+ q^{-m'+1})}{q (1-q^{-\lfloor \frac{n+n'}{2}\rfloor+1})}.
  \end{align*}  
  We will show first that $c\vcentcolon=X_1+X_2\le \frac{43}{16}$
  whenever $m+m'\ge3$, and consider smaller dimensional cases later
  in the proof.  Recall that $q\ge3$.
  
  Applying Lemma~\ref{L:boundprod} with
  $x=m$, $y=m'$, $\ell=m-1$ and $a=1$ gives
  \[
    (1+q^{-m})(1+q^{-m'})\le (1+q^{-1})(1+q^{-m-m'+1}),
  \]
  and applying Lemma~\ref{L:boundprod} with 
  $x=m-1$, $y=m'-1$, $\ell=m-1$ and $a=\frac{1}{1-q^{-1}}$ yields
  \[
  \left(a + q^{-m+1}\right)
  \left(a + q^{-m'+1}\right)
  \le  \left(a + 1\right)
     \left(a + q^{-m-m'+2}\right).
  \]
  Thus, if $m+m'\ge 3$, so that $\lfloor \frac{n+n'}{2}\rfloor\ge3$, then
  \begin{align*}
  c=X_1+X_2 &\le
  \frac{(1+q^{-1})(1+   q^{-2})}{1-q^{-3}} + 
  \frac{(1-q^{-1})(\frac{1}{1-q^{-1}}+1)(\frac{1}{1-q^{-1}} + q^{-1})}
     {q(1 - q^{-2})}\\
  &\le \frac{20}{13} + \frac{55}{48} < \frac{43}{16}.
  \end{align*}
  We may therefore assume that $m=m'=1$, so $n,n'\in\{2,3\}$.

  Consider the case $n=n'=3$.
  By Equation~\eqref{E:rho2} and
  Lemmas~\ref{L:gamma}, \ref{L:cP1} and~\ref{L:cN1} we have
  \begin{align*}
    q\varphi&\le c_1+\frac{c_2}{q}
    \le Z_1+\frac{\gamma_3^\circ(q)^2}{q(1-q^{-1})\gamma_6^{-}(q)}\\
    &=\frac{1+q^{-2}}{1-q^{-3}}
      +\frac{(1-q^{-2})^2}{q(1-q^{-1})(1-q^{-2})(1+q^{-3})}\\
    &=\frac{1+q^{-2}}{1-q^{-3}}
      +\frac{1+q^{-1}}{q(1+q^{-3})}\le\frac{15}{13}+\frac{3}{7}<\frac{43}{16}.
  \end{align*}
  Next suppose that $n=2$ and $n'=3$. Lemmas~\ref{L:gamma}, \ref{L:cP1}
  and~\ref{L:cN1} give
  \begin{align*}
    q\varphi&\le c_1+\frac{c_2}{q}
    \le Y_1+\frac{\gamma_2^{+}(q)\gamma_3^\circ(q)}{q(1-q^{-1})\gamma_5^{\circ}(q)}\\
    &=\frac{(1+q^{-1})(1+q^{-3})}{1-q^{-4}}
      +\frac{2(1-q^{-1})(1-q^{-2})}{q(1-q^{-1})(1-q^{-4})}\\
    &=\frac{1+q^{-3}}{(1-q^{-1})(1+q^{-2})}
      +\frac{2}{q(1+q^{-2})}\le\frac{7}{5}+\frac{3}{5}<\frac{43}{16}.
  \end{align*}
  A similar argument handles that case $n=3$ and $n'=2$.
  
  Finally, suppose that $n=n'=2$. If $q\ge4$, then
  Lemmas~\ref{L:gamma}, \ref{L:cP1} and~\ref{L:cN1} show that
  \begin{align*}
    q\varphi&\le c_1+\frac{c_2}{q}
    \le X_1+\frac{\gamma_2^{+}(q)^2}{q(1-q^{-1})\gamma_4^{-}(q)}\\
    &=\frac{(1+q^{-1})^2}{1-q^{-2}}
    +\frac{4(1-q^{-1})^2}{q(1-q^{-1})(1-q^{-1})(1+q^{-2})}
    \le\frac{5}{3}+\frac{16}{17}
    <\frac{43}{16}.
  \end{align*}
  For the remaining case $(n,n',q)=(2,2,3)$ the proportion
  $\varphi$ can be shown by using {\sc GAP}~\cites{gap,FinInG}
  to (easily) satisfy the bound $\varphi<43/16$.
\end{proof}

\begin{remark}\label{R:322}
When both $m=\lfloor n/2\rfloor$ and
  $m'=\lfloor n'/2\rfloor$ are large, the expressions for $X_1$ and $X_2$ in Lemmas~\ref{L:cP1} and~\ref{L:cN1} satisfy $\lim_{m,m'\to\infty}X_1=1$ and
  $\lim_{m,m'\to\infty}X_2=\frac{1}{q-1}$, and hence 
  $\lim_{m,m'\to\infty}(X_1+X_2)=1/(1-q^{-1})$. Hence, with the level of approximation we use for our estimates  it
  is impossible to bound  $1/(1-q^{-1})$ away from $q$ when $q=2$. To obtain the result for $q=2$ in the orthogonal case, we would need the next level of approximation embarked on for the symplectic case in Subsection~\ref{s:sp2} (and many more pages of argument as different isometry types must be considered). \hfill$\diamond$
\end{remark}

\begin{proof}[Proof of Theorem~$\ref{T:main1}$]
  The proof follows readily from the strategy outlined in
  Section~\ref{S:strategy} and
  Theorems~\ref{T:Sp}, \ref{T:U} and~\ref{T:O}.
\end{proof}


\section*{Acknowledgements}
We thank Jesse Lansdown for his help with computational investigations
using the  FinInG package  \cite{FinInG} in \textsf{GAP}.   The second
and  third   authors  thank  the  Hausdorff   Research  Institute  for
Mathematics,  University  of  Bonn,  for its  hospitality  during  the
International  Workshop on  Logic and  Algorithms in  Group Theory  in
2018.   We thank  Gabriele  Nebe for  helpful  discussions during  the
workshop.  The  authors gratefully acknowledge Australian Research
Council  Discovery Project Grant DP190100450.
The   second    author   acknowledges   funding   by    the   Deutsche
Forschungsgemeinschaft  (DFG,  German  Research  Foundation)  - Project
286237555 - TRR 195. We are grateful to the two referees for
their helpful suggestions.

\end{document}